\documentclass[11pt]{article}

\usepackage[T1]{fontenc}
\usepackage[utf8]{inputenc}
\usepackage{lmodern}
\usepackage[a4paper,margin=1in]{geometry}
\usepackage{graphicx}
\usepackage{multirow}
\usepackage{amsmath,amssymb,amsfonts}
\usepackage{amsthm}
\usepackage{bbm}
\usepackage{mathrsfs}
\usepackage[title]{appendix}
\usepackage{xcolor}
\usepackage{textcomp}
\usepackage{manyfoot}
\usepackage{booktabs}
\usepackage{algorithm}
\usepackage{algorithmicx}
\usepackage{algpseudocode}
\usepackage{listings}
\usepackage{float}
\usepackage{mathtools}
\usepackage{dsfont}
\usepackage[numbers,sort&compress]{natbib}
\usepackage[colorlinks=true,linkcolor=blue,citecolor=blue,urlcolor=blue]{hyperref}

\theoremstyle{plain}
\newtheorem{theorem}{Theorem}[section]
\newtheorem{proposition}{Proposition}[section]
\newtheorem{lemma}{Lemma}[section]
\newtheorem{corollary}{Corollary}[section]
\theoremstyle{definition}
\newtheorem{definition}{Definition}[section]
\newtheorem{example}{Example}[section]
\theoremstyle{remark}
\newtheorem{remark}{Remark}[section]

\newcommand{\ez}{ e_{\mathstrut 0}}

\newcommand{\sd}{s_{\mathstrut d}}
\def\multiset#1#2{\ensuremath{\left(\kern-.3em\left(\genfrac{}{}{0pt}{}{#1}{#2}\right)\kern-.3em\right)}}

\title{Discrete time-multidimensional renewal theory and applications}
\author{Leonidas Kordalis\\
Department of Mathematics, National and Kapodistrian University of Athens\\
Panepistimiopolis, Athens 15784, Greece\\
\texttt{lkordali@math.uoa.gr}
\and
Samis Trevezas\thanks{Corresponding author.}\\
Department of Mathematics, National and Kapodistrian University of Athens\\
Panepistimiopolis, Athens 15784, Greece\\
MICS Laboratory, CentraleSup\'elec -- Universit\'e Paris-Saclay\\
3 Rue Joliot Curie, Gif-sur-Yvette 91190, France\\
\texttt{strevezas@math.uoa.gr}}
\date{}

\begin{document}

\maketitle

\begin{abstract}
We develop a discrete-time renewal framework in which renewal events evolve along multiple time coordinates and the sojourn mechanism is described by a general distribution on the multi-index lattice. The resulting processes, called multi-time renewal chains, are studied through multi-index convolution and the associated algebra of multivariate formal power series. This algebraic formulation gives explicit representations for multi-time renewal equations, constructive coefficient formulas, and practical inversion schemes. For computation, we combine FFT-based multidimensional convolution with Newton-type reciprocal iteration to evaluate renewal quantities on large grids. For asymptotics, we prove strong laws and central limit theorems under proportional growth of the observation horizon, including a general central limit theorem for additive functionals and a Gaussian limit for the renewal counting process in directions with a unique rate-determining coordinate. We also study fixed-horizon observations: the terminal age vector induces a genuinely multivariate right-censoring mechanism, leading to an exact nonparametric maximum likelihood estimator and its asymptotic normality. Applications include a binomial--multiset identity, two-attribute warranty evaluation, alternating-renewal availability computation, and discretization-based approximations of continuous-time bivariate renewal and availability models.
\end{abstract}


\section{Introduction}

Renewal models describe systems that ``restart'' after random sojourn times and are a central tool in
applied probability, reliability, and operations research. In many modern applications, however, the
evolution of a system is not naturally indexed by a single time axis: events may be driven
simultaneously by multiple exposure dimensions (e.g.\ calendar age and cumulative usage, operating time
and stress cycles, or workload and environmental conditions). This motivates renewal constructions in
which renewal epochs evolve on a multi-index time grid, and also provides a basic building block for
state-dependent generalizations such as multi-time Markov renewal chains.

The theory of renewal equations in two-dimensional \emph{continuous} time was introduced and developed by
Hunter in a series of papers \cite{hunter1974renewal,hunter1974renewal1,hunter1977renewal}. In that setting,
the bivariate renewal function satisfies a renewal \emph{integral} equation and admits transform
representations, bounds, and asymptotic descriptions. Hunter's framework has proved particularly
influential in reliability, where failure and repair mechanisms may depend on coexisting exposure
variables such as calendar and operational time, or cumulative usage and load cycles
\cite{insua2020advances,jack2009repair,lai2006stochastic,yang2001bivariate}. A prominent example is that of
two-attribute warranty policies, where coverage is constrained by both elapsed time since purchase and
usage of the product \cite{kim2000expected,murthy1995two}.

Beyond the foundational bivariate theory, several authors have aimed to broaden the scope of
multi-time renewal analysis. On the theoretical side, multivariate renewal processes and their
fluctuations have been investigated in probability and statistics; for instance, asymptotic
distributional results for multivariate renewal counts were derived in \cite{niculescu1984asymptotic}.
On the applied side, multi-attribute replacement and maintenance questions naturally lead to \emph{weighted}
multivariate renewal functions and their associated renewal equations \cite{mallor2007multivariate}.
On the computational side, even in two dimensions the renewal function rarely has a closed form, and a
number of approximations have been proposed, including moment-based and discretization-based schemes
\cite{hadji2015two,arunachalam2015approximation}. On the inference side, nonparametric estimation of
two-dimensional/bivariate renewal functions has recently been developed in \cite{harel2019asymptotic}.
Despite this progress, two persistent challenges remain: (i) scaling computations to large multi-index
grids (and beyond the bivariate case), and (ii) connecting renewal \emph{equations} with constructive,
coefficient-level representations that can be exploited both theoretically and algorithmically.

In this work we develop a discrete time-multidimensional renewal framework on $\mathbb{N}^d$. We consider
a renewal chain
\begin{equation*}
S_n=\sum_{i=1}^n X_i,\qquad n\ge 1,
\end{equation*}
where $(X_i)_{i\ge 1}$ is a sequence of i.i.d.\ $d$-dimensional random vectors with nonnegative integer
components. The associated multi-time counting process at $k\in\mathbb{N}^d$ is
\begin{equation*}
N(k)=\sup\{n\in\mathbb{N}: S_n\le_d k\},
\end{equation*}
where $\le_d$ denotes componentwise order. This construction captures the number of renewals occurring
inside the hyper-rectangle $\prod_{u=1}^d[0,k_u]$ and provides a natural discrete analogue of continuous-time
multivariate renewal models. Besides being directly relevant when exposure variables are recorded in
discrete units, the lattice setting also serves as a convenient computational interface for
discretization-based approximations of continuous-time multi-time renewal quantities.

A central tool in our development is the multi-index convolution product of real sequences
on $\mathbb{N}^d$. By identifying sequences with multivariate formal power series, convolution becomes
ordinary multiplication of power series. This viewpoint yields a unified algebraic language for renewal
equations of the form $g=b+f*g$, and leads to explicit solution representations via convolutional inverses
such as $(\ez-f)^{(-1)}$ whenever they exist, where $\ez$ denotes the convolutional identity element. The same convolutional
structure is useful computationally: it enables FFT-accelerated (Fast Fourier Transform) evaluation of multidimensional
convolutions and supports fast inversion strategies based on truncated series arithmetic (Newton-type
reciprocal iteration). These tools are needed in multi-time settings, where direct
summation rapidly becomes prohibitive as the grid size and the dimension increase.

The main contributions of the paper are summarized as follows:
\begin{itemize}
  \item \textbf{Discrete multi-time renewal chains and renewal equations.}
  We introduce a discrete-time $d$-dimensional renewal framework and derive multi-time renewal equations
  for key quantities, including the renewal function and reliability/availability functionals in multi-index time.
  This provides a lattice counterpart to the classical bivariate continuous-time renewal equations
  \cite{hunter1974renewal,hunter1974renewal1}.

  \item \textbf{Algebraic properties of multi-index convolution.}
  We develop basic properties of the convolution algebra of sequences on $\mathbb{N}^d$, including existence
  and constructive representations of convolutional inverses, and finite-sum/recursive formulas for inverse coefficients.
  The resulting coefficient-level formulas can be used both analytically and as building blocks for numerical routines.

  \item \textbf{Efficient computation via FFT and Newton-type inversion.}
  We discuss FFT-based convolution techniques for high-dimensional grids and adapt Newton-type reciprocal iteration
  for computing convolutional inverses in a truncated multivariate power series setting. This gives a scalable
  computational procedure to approaches that approximate renewal quantities indirectly (e.g.\ via moments or coarse
  discretizations in the bivariate continuous-time setting \cite{hadji2015two,arunachalam2015approximation}).

  \item \textbf{Asymptotic results under proportional scaling.}
  Under standard moment conditions and a proportional growth regime $k/|k|_1\to\lambda$, we establish a strong law
  for $N(k)$ and the multi-time renewal function $M(k)$ and a CLT for additive functionals. We also obtain a Gaussian limit for $N(k)$ in directions
  where a unique rate-determining coordinate governs the minimum of the marginal counts, complementing asymptotic
  results in the continuous-time bivariate theory \cite{hunter1974renewal1} and in multivariate renewal asymptotics
  \cite{niculescu1984asymptotic}.

  \item \textbf{Inference from fixed-horizon data.}
  From a single trajectory observed up to a multi-time horizon $m$, we study maximum likelihood estimation of the
  increment law under multivariate right-censoring encoded by the age vector $U_m$. We derive the exact
  nonparametric MLE and establish its asymptotic normality. This addresses a complementary inferential problem to
  existing work on estimating renewal \emph{functions} in multi-time settings \cite{harel2019asymptotic}.

  \item \textbf{Applications.}
  We illustrate the framework in (i)~bivariate reliability/availability modeling for alternating renewal systems,
  (ii)~two-attribute warranty policies, (iii)~a binomial--multiset identity implied by the inverse-coefficient
  formula, and (iv)~discrete approximations of continuous-time bivariate renewal and availability models.
\end{itemize}

The remainder of the paper is organized as follows. Section~\ref{sec:convolution} develops the convolutional
framework and the associated computational tools. Section~\ref{sec:multitime-renewal-theory} introduces multi-time
renewal chains and derives the corresponding renewal equations. Section~\ref{sec:limit} presents asymptotic results
for the counting process and for additive functionals. Section~\ref{sec:mle} develops maximum likelihood estimation
from fixed-horizon observations under multivariate censoring. Section~\ref{sec:Applications} provides numerical and
applied illustrations in reliability, warranty analysis, and continuous-time approximation. The appendices contain
implementation details for FFT-based convolution and a proof of the Newton-type inversion update.

\section{Multi-index convolution algebra and inversion}\label{sec:convolution}

This section develops the algebraic tools that will be used throughout the paper to formulate and solve multi-time renewal equations.
We begin by reviewing the multi-index convolution product of real sequences on $\mathbb{N}^d$ and its representation via multivariate
formal power series. This viewpoint yields an explicit solution of renewal-type convolution equations through convolutional inverses,
and it also provides a natural route to efficient numerical implementation (via FFT-based convolution and fast series inversion).
The probabilistic multi-time renewal model and its counting process are introduced in
Section~\ref{sec:multitime-renewal-theory}, where these algebraic tools are applied to derive
renewal equations for quantities of interest.

\subsection{Discrete multi-time convolution product}\label{subsec:multi-conv}

\paragraph{Notation.}
Throughout, $d\in\mathbb{N}$ is fixed and we use $\mathbb{N}:=\{0,1,2,\ldots\}$.
For $k=(k_1,\ldots,k_d)\in\mathbb{N}^d$ we write $|k|_1:=k_1+\cdots+k_d$.
We denote by $0_d$ the null vector and by $1_d$ the vector of ones in $\mathbb{R}^d$.
For $k,\ell\in\mathbb{N}^d$, the notation $\ell\le_d k$ means componentwise inequality.
When needed (e.g.\ in finite-difference identities), we adopt the standard convention that a sequence $a$ defined on $\mathbb{N}^d$
is extended by $0$ outside $\mathbb{N}^d$, i.e.\ $a(k)=0$ if some component of $k$ is negative.

\subsubsection{Formal power series and multi-index convolution}\label{subsubsec:fps-conv}

Let $a:\mathbb{N}^d\to\mathbb{R}$ be a real $d$-time sequence. We associate to $a$ the multivariate formal power series
\begin{equation}\label{eq:fps}
P_a(x_{1:d})=\sum_{k\in\mathbb{N}^d} a(k)\,x^k,
\qquad
x^k:=\prod_{u=1}^d x_u^{k_u}.
\end{equation}
Here, $\mathbb{R}[[x_1,\ldots,x_d]]$ denotes the ring of formal power series in $d$ indeterminates with real coefficients.
Coefficientwise addition of sequences corresponds to addition of formal power series. The multiplicative structure is given by the
multi-index (Cauchy) product, i.e.\ convolution.

\noindent We first fix the convolution product on the multi-index lattice.
\begin{definition}\label{def:conv}
Let $a,b:\mathbb{N}^d\to\mathbb{R}$. Their (multi-time) convolution product $c=a*b$ is defined by
\begin{equation}\label{eq:conv}
[a*b](k)
:=\sum_{\ell+\ell'=k} a(\ell)\,b(\ell')
=\sum_{\ell\le_d k} a(\ell)\,b(k-\ell),
\qquad k\in\mathbb{N}^d,
\end{equation}
where $\le_d$ denotes componentwise inequality.
\end{definition}

\noindent
\emph{Well-definedness.} For each fixed $k\in\mathbb{N}^d$, the index set
$\{\ell\in\mathbb{N}^d:\ell\le_d k\}$ is finite (of cardinality $\prod_{u=1}^d (k_u+1)$). Hence \eqref{eq:conv} is a finite sum and
$a*b$ is well-defined on $\mathbb{R}^{\mathbb{N}^d}$ without additional summability assumptions.

\noindent The algebraic structure and its power-series representation are recorded in the following proposition.
\begin{proposition}\label{prop:ring}
Let $\mathcal{R}:=\mathbb{R}^{\mathbb{N}^d}$ be the set of all real-valued sequences on $\mathbb{N}^d$.
Equipped with pointwise addition $+$ and the convolution product $*$ in \eqref{eq:conv}, $(\mathcal{R},+,*)$ is a commutative ring
with unity. The identity element (unit) is the sequence $\ez$ defined by
\[
\ez(k)=
\begin{cases}
1, & k=0_d,\\
0, & \text{otherwise}.
\end{cases}
\]
If there exists $b:\mathbb{N}^d\to\mathbb{R}$ such that $a*b=b*a=\ez$, then $a$ is \emph{convolutionally invertible} and $b$ is called
its \emph{convolutional inverse}, denoted by $a^{(-1)}$.

Moreover, the map $\Phi:\mathcal{R}\to \mathbb{R}[[x_1,\ldots,x_d]]$ defined by $\Phi(a)=P_a$ is a ring isomorphism. In particular,
for all $a,b\in\mathcal{R}$,
\[
P_{a+b}(x_{1:d})=P_a(x_{1:d})+P_b(x_{1:d}),\ \
P_{a*b}(x_{1:d})=P_a(x_{1:d})\,P_b(x_{1:d}),\ \
P_{\ez}(x_{1:d})=1.
\]
\end{proposition}

\paragraph{Convolution powers.}
For $a\in\mathcal{R}$ and $n\in\mathbb{N}$ we define the $n$-fold convolution power by
\[
a^{(0)}:=\ez,
\qquad
a^{(n)}:=a*a^{(n-1)},\quad n\ge 1.
\]
Equivalently, for $n\ge 1$ and $k\in\mathbb{N}^d$,
\begin{equation}\label{eq:nfold}
a^{(n)}(k)=
\sum_{\substack{\ell_1,\ldots,\ell_n\in\mathbb{N}^d\\ \ell_1+\cdots+\ell_n=k}}
a(\ell_1)\cdots a(\ell_n).
\end{equation}
Since $(\mathcal{R},+,*)$ is a commutative ring with unity, the binomial theorem holds for convolution powers:
\begin{equation*}
(a+b)^{(n)}=\sum_{j=0}^{n}\binom{n}{j}\, a^{(j)}*b^{(n-j)},
\qquad n\in\mathbb{N}.
\end{equation*}
In particular,
\begin{equation}\label{eq:binomial-e0}
(\ez-a)^{(n)}=\sum_{j=0}^{n}(-1)^j\binom{n}{j}\,a^{(j)}.
\end{equation}

\subsubsection{A finite-support property for convolution powers}

\begin{lemma}\label{lem:vanishing-powers}
Let $c:\mathbb{N}^d\to\mathbb{R}$ satisfy $c(0_d)=0$.
Then for every $k\in\mathbb{N}^d$,
\begin{equation*}
c^{(n)}(k)=0\qquad\text{for all }n>|k|_1.
\end{equation*}
\end{lemma}

\begin{proof}
Fix $k\in\mathbb{N}^d$ and $n\ge 1$.
In \eqref{eq:nfold}, any summand with some $\ell_i=0_d$ vanishes because $c(0_d)=0$.
Hence any potentially nonzero term must satisfy $\ell_i\neq 0_d$ for all $i$, which implies $|\ell_i|_1\ge 1$.
Therefore, for any nonzero summand,
\begin{equation*}
|k|_1=\Big|\sum_{i=1}^n \ell_i\Big|_1=\sum_{i=1}^n |\ell_i|_1\ge n.
\end{equation*}
If $n>|k|_1$, there are no such decompositions and thus $c^{(n)}(k)=0$.
\end{proof}

\subsubsection{Convolutional inverse: existence and representations}

\noindent The following theorem characterizes the invertible elements and gives a coefficientwise inverse.
\begin{theorem}\label{thm:inverse-exists}
Let $a:\mathbb{N}^d\to\mathbb{R}$.
Then $a$ is convolutionally invertible if and only if $a(0_d)\neq 0$.
Moreover, if $a(0_d)\neq 0$ and we define the normalized sequence
\begin{equation}\label{eq:normalize-a}
\widetilde a(k):=\frac{a(k)}{a(0_d)},\qquad k\in\mathbb{N}^d,
\end{equation}
then $\widetilde a(0_d)=1$ and the convolutional inverse of $a$ is given coefficientwise by
\begin{equation}\label{eq:inverse-geometric}
a^{(-1)}(k)
=\frac{1}{a(0_d)}\sum_{n=0}^{|k|_1}\big(\ez-\widetilde a\big)^{(n)}(k),
\qquad k\in\mathbb{N}^d.
\end{equation}
Equivalently,
\begin{equation}\label{eq:inverse-binomial}
a^{(-1)}(k)
=\frac{1}{a(0_d)}
\sum_{n=0}^{|k|_1}
(-1)^n
\binom{|k|_1+1}{n+1}\,
\widetilde a^{(n)}(k),
\qquad k\in\mathbb{N}^d.
\end{equation}
\end{theorem}

\begin{proof}
If $a^{(-1)}$ exists, then evaluating $a*a^{(-1)}=\ez$ at $0_d$ yields $a(0_d)\,a^{(-1)}(0_d)=1$, hence $a(0_d)\neq 0$.

Conversely, assume $a(0_d)\neq 0$ and define $\widetilde a$ by \eqref{eq:normalize-a}.
Set $c:=\ez-\widetilde a$, so that $c(0_d)=0$.
Define $b:\mathbb{N}^d\to\mathbb{R}$ by
\begin{equation*}
b(k):=\sum_{n=0}^{|k|_1} c^{(n)}(k).
\end{equation*}
This sum is finite for each fixed $k$ by Lemma~\ref{lem:vanishing-powers}.
Using associativity and distributivity of convolution,
\begin{equation*}
(\ez-c)*b
=
\sum_{n=0}^{\infty} c^{(n)}-\sum_{n=0}^{\infty} c^{(n+1)}
=\ez,
\end{equation*}
where the equalities are to be understood coefficientwise (and are legitimate because every coefficient involves only finitely many
nonzero terms by Lemma~\ref{lem:vanishing-powers}). Since $\ez-c=\widetilde a$, we get $\widetilde a*b=\ez$, hence
$a*\big(\frac{1}{a(0_d)}b\big)=\ez$, proving \eqref{eq:inverse-geometric}.

To obtain \eqref{eq:inverse-binomial}, expand $c^{(m)}=(\ez-\widetilde a)^{(m)}$ using \eqref{eq:binomial-e0} and sum over $m=0,\ldots,|k|_1$.
The combinatorial identity $\sum_{m=n}^{r}\binom{m}{n}=\binom{r+1}{n+1}$ (``hockey-stick'' \cite{jones1996generalized}) yields the stated coefficients.
\end{proof}

\subsubsection{Examples}\label{subsubsec:examples}

Throughout the examples below we use the standard convention that a sequence $a$ defined on $\mathbb{N}^d$ is extended by $0$ outside
$\mathbb{N}^d$, i.e.\ $a(k)=0$ whenever some component of $k$ is negative. This is convenient for writing finite-difference identities.

\noindent The summation sequence and its inverse give the following finite-difference example.
\begin{example}\label{ex:summation-difference}
Let $\sd$ denote the constant-one sequence on $\mathbb{N}^d$, i.e.\ $\sd(k)=1$ for all $k\in\mathbb{N}^d$.
Then, for any $a:\mathbb{N}^d\to\mathbb{R}$,
\begin{equation}\label{eq:multi-sum}
(\sd*a)(k)=\sum_{\ell\le_d k} a(\ell),\qquad k\in\mathbb{N}^d.
\end{equation}
Indeed, \eqref{eq:multi-sum} is \eqref{eq:conv} with $\sd(\ell)\equiv 1$, followed by the change of variables $m=k-\ell$.

The generating function of $\sd$ is
\begin{equation*}
P_{\sd}(x_{1:d})=\sum_{k\in\mathbb{N}^d}x^k=\prod_{u=1}^d \frac{1}{1-x_u}.
\end{equation*}
Hence $\sd$ is convolutionally invertible and its inverse $\delta_d:=\sd^{(-1)}$ satisfies
\begin{equation*}
P_{\delta_d}(x_{1:d})=\prod_{u=1}^d (1-x_u)
=
\sum_{\ell\in\{0,1\}^d}(-1)^{|\ell|_1}\,x^\ell.
\end{equation*}
Equivalently, $\delta_d(\ell)=(-1)^{|\ell|_1}$ for $\ell\in\{0,1\}^d$ and $\delta_d(\ell)=0$ otherwise.
Therefore, for any $a$ and any $k\in\mathbb{N}^d$,
\begin{equation}\label{eq:multi-diff}
(\delta_d*a)(k)=\sum_{\ell\le_d 1_d} (-1)^{|\ell|_1}\, a(k-\ell),
\end{equation}
where the extension-by-zero convention ensures that terms with negative indices vanish.
For $d=2$, \eqref{eq:multi-diff} reduces to the usual inclusion--exclusion finite difference:
\begin{equation*}
(\delta_2*a)(k_1,k_2)=a(k_1,k_2)-a(k_1-1,k_2)-a(k_1,k_2-1)+a(k_1-1,k_2-1).
\end{equation*}
\end{example}

\noindent The one-dimensional case gives the following binomial and multiset coefficients.
\begin{example}\label{ex:binomial-multiset}
In one dimension ($d=1$), let $b_1=(1,1,0,0,\ldots)$ so that $P_{b_1}(x)=1+x$.
Then $P_{b_1^{(n)}}(x)=(1+x)^n$, hence $b_1^{(n)}(k)=\binom{n}{k}$.

Similarly, let $s_1:\mathbb{N}\to\mathbb{R}$ be the constant-one sequence $s_1(k)\equiv 1$, so that
$P_{s_1}(x)=1+x+x^2+\cdots=(1-x)^{-1}$.
For $n\ge 1$,
\begin{equation*}
P_{s_1^{(n)}}(x)=P_{s_1}(x)^n=(1-x)^{-n}
=\sum_{k\ge 0}\binom{n+k-1}{k}\,x^k,
\end{equation*}
and therefore
\begin{equation*}
s_1^{(n)}(k)=\binom{n+k-1}{k}=\multiset{n}{k}
=\frac{n(n+1)\cdots(n+k-1)}{k!}.
\end{equation*}
The coefficient $\multiset{n}{k}$ is the number of $k$-combinations from an $n$-element set \emph{with repetition}
(the classical stars-and-bars interpretation).
\end{example}

\noindent The next proposition gives the convolution powers of the constant-one sequence.
\begin{proposition}\label{prop:sd-nfold}
For $n\ge 1$ and $k=(k_1,\ldots,k_d)\in\mathbb{N}^d$,
\begin{equation*}
\sd^{(n)}(k)=\prod_{u=1}^d \multiset{n}{k_u}
=\prod_{u=1}^d \binom{n+k_u-1}{k_u}.
\end{equation*}
\end{proposition}

\begin{proof}
Using generating functions,
\[
P_{\sd^{(n)}}(x_{1:d})=P_{\sd}(x_{1:d})^n
=\prod_{u=1}^d (1-x_u)^{-n}
=\prod_{u=1}^d \left(\sum_{k_u\ge 0}\multiset{n}{k_u}\,x_u^{k_u}\right).
\]
Extracting the coefficient of $x^k=\prod_{u=1}^d x_u^{k_u}$ yields the product form.
\end{proof}

\noindent The preceding formula has the following combinatorial interpretation.
\begin{remark}\label{rem:sd-nfold-comb}
Since $\sd(\ell)\equiv 1$, the definition \eqref{eq:nfold} shows that $\sd^{(n)}(k)$ equals the number of $n$-tuples
$(\ell_1,\ldots,\ell_n)\in(\mathbb{N}^d)^n$ such that $\ell_1+\cdots+\ell_n=k$ (componentwise), i.e.\ the number of weak compositions
of the multi-index $k$ into $n$ parts. Equivalently, $\sd^{(n)}(k)$ counts the number of $n\times d$ nonnegative integer arrays
$(x_{j,u})$ such that $\sum_{j=1}^n x_{j,u}=k_u$ for each coordinate $u$.

This admits a natural ``multi-type'' stars-and-bars interpretation:
for each type $u\in\{1,\ldots,d\}$, distribute $k_u$ identical balls (stars of color $u$) into $n$ labeled bins.
The number of such allocations is $\binom{n+k_u-1}{k_u}$ for each $u$, and independence across types yields the product in
Proposition~\ref{prop:sd-nfold}.

\smallskip
\noindent
\emph{Caution.} If one instead counts \emph{strings} formed by $n-1$ bars together with $k_u$ stars of color $u$ (all symbols arranged in
a single line), then the count is the multinomial coefficient
\[
\binom{n+\sum_{u=1}^d k_u-1}{\,n-1,\;k_1,\ldots,k_d\,}
=\frac{\big(n+\sum_{u=1}^d k_u-1\big)!}{(n-1)!\,k_1!\cdots k_d!},
\]
which records the within-bin interleaving order of colors and is generally larger than $\prod_{u=1}^d \binom{n+k_u-1}{k_u}$.
For instance, when $n=2$ and $k_1=k_2=1$, the multinomial count equals $6$, whereas the product count equals $4$.
\end{remark}

\noindent The convolution notation also expresses the relation between a discrete distribution and its cdf.
\begin{example}\label{ex:cdf-pmf}
Let $f:\mathbb{N}^d\to[0,1]$ be a pmf and define the multi-index cdf
\begin{equation*}
F(k):=\sum_{\ell\le_d k} f(\ell)=\mathbb{P}(X\le_d k),\qquad k\in\mathbb{N}^d.
\end{equation*}
Then $F=\sd*f$ by \eqref{eq:multi-sum}.
Conversely, with $\delta_d$ from Example~\ref{ex:summation-difference}, one has $f=\delta_d*F$, i.e.
\begin{equation*}
f(k)=\sum_{\ell\le_d 1_d} (-1)^{|\ell|_1}\,F(k-\ell),\qquad k\in\mathbb{N}^d,
\end{equation*}
with the convention that $F(k)=0$ if some component of $k$ is negative.
This is the discrete analogue of recovering a density from a cdf via mixed partial derivatives.
\end{example}

\noindent The same notation gives the distribution of sums of independent random vectors.
\begin{example}\label{ex:sums-iid}
Let $X_1,\ldots,X_n$ be i.i.d.\ $\mathbb{N}^d$-valued random vectors with pmf $f$ and cdf $F$ (with respect to $\le_d$), and set
$S_n:=X_1+\cdots+X_n$.
Then the pmf of $S_n$ is $f^{(n)}$.
Moreover, its cdf is
\begin{equation*}
\mathbb{P}(S_n\le_d k)=\sum_{\ell\le_d k} f^{(n)}(\ell)=(\sd*f^{(n)})(k),\qquad k\in\mathbb{N}^d.
\end{equation*}
Using $F=\sd*f$ and associativity/commutativity of convolution, we also obtain
\begin{equation*}
\sd*f^{(n)}=(\sd*f)*f^{(n-1)}=F*f^{(n-1)}=f^{(n-1)}*F,
\qquad n\ge 1.
\end{equation*}
\end{example}

\subsubsection{Computational remarks: convolution and inversion}\label{subsubsec:computational}

We briefly comment on the numerical evaluation of multi-index convolutions and convolutional inverses on finite grids.
Throughout this subsection we consider real-valued sequences supported on finite boxes in $\mathbb{N}^d$.
Fast convolution methods are classical and rely on the FFT \cite{cooley1965algorithm,van1992computational}, while fast power-series inversion via
Newton iteration is standard in computer algebra \cite{brent1978fast,von2013modern}.

\noindent The two basic computational regimes are as follows.
\begin{remark}\label{rem:fft-cost}
Let $a$ and $b$ be supported on the finite boxes
\begin{equation*}
\mathrm{supp}(a)\subseteq \prod_{u=1}^d \{0,\ldots,K_u\},
\qquad
\mathrm{supp}(b)\subseteq \prod_{u=1}^d \{0,\ldots,M_u\}.
\end{equation*}
Then $c:=a*b$ is supported on $\prod_{u=1}^d \{0,\ldots,K_u+M_u\}$.

\smallskip
\noindent\emph{Direct computation.}
Computing all coefficients of $c$ on its full support by the definition \eqref{eq:conv} requires on the order of
\[
\mathcal{O}\!\left(\prod_{u=1}^d (K_u+1)(M_u+1)\right)
\]
arithmetic operations, i.e.\ quadratic in the total number of input coefficients.
In the symmetric case $K_u=M_u$, we get the order $\mathcal{O}\!\left(\prod_{u=1}^d (K_u+1)^2\right)$.

\smallskip
\noindent\emph{FFT-based computation.}
Let $L_u$ be padding lengths satisfying $L_u\ge K_u+M_u+1$ for all $u$, and embed $a$ and $b$ into $d$-dimensional arrays of size
$L_1\times\cdots\times L_d$ by zero-padding.
Then linear convolution can be computed via FFT by exploiting that FFT multiplication corresponds to \emph{circular} convolution
on the padded grid; the padding condition prevents wrap-around (aliasing) and ensures that the circular convolution equals the
desired linear convolution on $\{0,\ldots,K_1+M_1\}\times\cdots\times\{0,\ldots,K_d+M_d\}$.

Let $N_L:=\prod_{u=1}^d L_u$.
A standard separable $d$-dimensional FFT (implemented as successive 1D FFTs along each axis) evaluates the transform and its inverse
in
\[
\mathcal{O}\!\left(N_L\sum_{u=1}^d \log L_u\right)
\]
operations; thus convolution via two forward FFTs, pointwise multiplication, and one inverse FFT has the same order of complexity
\cite{cooley1965algorithm,van1992computational}.
\end{remark}

\noindent The convolutional inverse can also be computed recursively.
\begin{proposition}\label{prop:inv-recursion}
Let $a:\mathbb{N}^d\to\mathbb{R}$ satisfy $a(0_d)\neq 0$ and set $b:=a^{(-1)}$.
Then
\[
b(0_d)=\frac{1}{a(0_d)},
\qquad
b(k)=-\frac{1}{a(0_d)}\sum_{\substack{0_d\le_d \ell <_d k}} b(\ell)\,a(k-\ell),
\quad k\neq 0_d,
\]
where $\ell<_d k$ means $\ell\le_d k$ and $\ell\neq k$.
\end{proposition}

\begin{proof}
The identity $a*b=\ez$ gives, for $k\neq 0_d$,
\[
0=(a*b)(k)=a(0_d)\,b(k)+\sum_{\substack{0_d\le_d \ell <_d k}} a(k-\ell)\,b(\ell),
\]
which yields the recursion after rearranging.
\end{proof}

\noindent The partial order in the recursion is handled as follows.
\begin{remark}\label{rem:inv-rec-cost}
Although $<_d$ is only a partial order, the recursion is practically computable by processing indices in increasing total degree:
if $\ell<_d k$, then necessarily $|\ell|_1<|k|_1$, hence all terms on the right-hand side are available once coefficients of total
degree $<|k|_1$ have been computed.
Computing $b(k)$ for all $k$ in a box of $N$ grid points by this direct recursion typically requires $\mathcal{O}(N^2)$ operations.
This motivates the use of FFT-accelerated Newton iteration when large grids are required.
\end{remark}

\paragraph{Newton-type inversion in multivariate formal series (total-degree truncation).}
Let $\mathfrak{m}:=\langle x_1,\ldots,x_d\rangle$ be the maximal ideal of $\mathbb{R}[[x_1,\ldots,x_d]]$.
For $m\ge 1$ and a series $P(x_{1:d})=\sum_k c_k x^k$, we write
\[
P \bmod \mathfrak{m}^m
\quad\text{to denote truncation to total degree }|k|_1<m.
\]
Let $P_a$ be as in \eqref{eq:fps} with $a(0_d)\neq 0$, and seek $P_b$ such that $P_aP_b=1$.
A natural initialization is $P_{b,1}:=a(0_d)^{-1}$, which satisfies $P_aP_{b,1}=1 \bmod \mathfrak{m}$.

Assume inductively that we have a truncated approximation $P_{b,m}$ satisfying
\[
P_a P_{b,m}=1 \quad \bmod \mathfrak{m}^m,
\]
and define the error $E_m:=1-P_aP_{b,m}\in \mathfrak{m}^m$.
Then the Newton update
\begin{equation}\label{eq:newton-multi}
P_{b,2m}
=
P_{b,m}+P_{b,m}\,E_m
=
P_{b,m}\,(2-P_a P_{b,m})
\quad \bmod \mathfrak{m}^{2m}
\end{equation}
produces an approximation that is correct modulo $\mathfrak{m}^{2m}$ (i.e.\ the truncation order doubles).
Indeed, one checks that $P_aP_{b,2m}=1-E_m^2$ and $E_m^2\in\mathfrak{m}^{2m}$.
This is the standard Newton iteration for reciprocal series; see, e.g., \cite{brent1978fast,von2013modern}.

\noindent The Newton step can be combined with FFT multiplication in the following way.
\begin{remark}\label{rem:newton-fft}
To implement \eqref{eq:newton-multi}, one must multiply truncated multivariate series.
When working modulo $\mathfrak{m}^{2m}$ (total-degree truncation), a simple implementation is to embed the truncated series into the
rectangular box of exponent vectors $\{0,\ldots,2m-1\}^d$, perform the polynomial products via $d$-dimensional FFT-based convolution on
a suitably padded grid as in Remark~\ref{rem:fft-cost}, and then discard all coefficients of total degree $\ge 2m$.
This approach is straightforward, although it computes some coefficients that are later discarded.
\end{remark}

\section{Multi-time renewal chains and renewal equations}\label{sec:multitime-renewal-theory}

This section introduces the discrete-time $d$-dimensional renewal model and derives the associated renewal equations.
The presentation is the discrete multi-time analogue of the classical univariate renewal framework (see, e.g.,
\cite{cox1962renewal,feller1971introduction,asmussen2003applied}), and is closely related in spirit to Hunter's continuous-time bivariate renewal theory
\cite{hunter1974renewal,hunter1974renewal1,hunter1977renewal}.

\subsection{Renewal chain and counting process}\label{subsec:renewal-chain-counting}

Let $(X_n)_{n\ge 1}$ be an i.i.d.\ sequence of $\mathbb{N}^d$-valued random vectors,
$X_n=(X_n^{[u]})_{1\le u\le d}$, with pmf $f$ and cdf (with respect to the componentwise order $\le_d$)
\[
f(k):=\mathbb{P}(X_1=k),\qquad
F(k):=\mathbb{P}(X_1\le_d k),\qquad k\in\mathbb{N}^d.
\]

We assume throughout that the increment is non-degenerate in the sense that $f(0_d)<1$.
In several places (in particular, when using finite-sum representations) it is convenient to
adopt the stronger condition
\begin{equation}\label{eq:no-zero-vector}
f(0_d)=0,
\end{equation}
i.e.\ $\mathbb{P}(X_1=0_d)=0$.
Note that \eqref{eq:no-zero-vector} only excludes simultaneous zero increments in all
coordinates; it still allows $\mathbb{P}(X_1^{[u]}=0)>0$ for individual coordinates.
Under \eqref{eq:no-zero-vector}, each renewal increases the $\ell^1$-level by at least one,
and consequently $N(k)\le |k|_1$ almost surely for every $k\in\mathbb{N}^d$.

To ensure that the marginal counting processes $N^{[u]}(k_u)$ are almost surely finite,
we will also assume that $\mathbb{P}(X_1^{[u]}>0)>0$ for each $u\in\{1,\ldots,d\}$ (equivalently, $S_n^{[u]}\to\infty$ almost surely.)

Define the renewal (jump) chain by
\begin{equation*}
S_0=0_d,\qquad S_n=\sum_{i=1}^n X_i,\qquad n\ge 1.
\end{equation*}
For a multi-index $k\in\mathbb{N}^d$, the associated multi-time counting process is
\begin{equation}\label{eq:N-def}
N(k):=\sup\{n\in\mathbb{N}: S_n\le_d k\},
\end{equation}
and its expectation
\[
M(k):=\mathbb{E}[N(k)]
\]
is called the multi-time renewal function.

For each coordinate $u\in\{1,\ldots,d\}$ we also define the marginal counting process
\[
N^{[u]}(k_u):=\sup\{n\in\mathbb{N}: S_n^{[u]}\le k_u\},\qquad k_u\in\mathbb{N}.
\]
Since $S_n\le_d k$ if and only if $S_n^{[u]}\le k_u$ for all $u$, it follows that
\begin{equation}\label{eq:N-min-marginals}
N(k)=\min_{1\le u\le d} N^{[u]}(k_u),\qquad k\in\mathbb{N}^d.
\end{equation}

\begin{remark}\label{rem:N-S}
From \eqref{eq:N-def} we have the equivalence
\[
N(k)\ge n \quad \Longleftrightarrow \quad S_n\le_d k,\qquad n\in\mathbb{N},\ k\in\mathbb{N}^d.
\]
\end{remark}

\subsection{Distributions of $S_n$ and $N(k)$}\label{subsec:dist-identities}

Let $f^{(n)}$ denote the $n$-fold convolution power of $f$ (with $f^{(0)}=\ez$).
Then $f^{(n)}$ is the pmf of $S_n$.
We denote by
\[
F_n(k):=\mathbb{P}(S_n\le_d k),\qquad k\in\mathbb{N}^d,
\]
the cdf of $S_n$.
Using conditioning on $S_{n-1}$ (or Example~\ref{ex:sums-iid}), for $n\ge 1$ we have the convolution identity
\begin{equation}\label{eq:F-n-conv}
F_n = f^{(n-1)}*F.
\end{equation}
\noindent The distribution of the multi-time count is obtained from the distributions of the renewal epochs.
\begin{proposition}\label{prop:dist-N}
Fix $k\in\mathbb{N}^d$. Then for all $n\in\mathbb{N}$,
\[
\overline{F}_{N(k)}(n):=\mathbb{P}(N(k)\ge n)=F_n(k),
\]
and consequently
\begin{eqnarray*}
F_{N(k)}(n)&:=&\mathbb{P}(N(k)\le n)=1-F_{n+1}(k),\\
f_{N(k)}(n)&:=&\mathbb{P}(N(k)=n)=F_n(k)-F_{n+1}(k).
\end{eqnarray*}
\end{proposition}

\begin{proof}
By Remark~\ref{rem:N-S},
\begin{equation}\label{surv-mean}
 \mathbb{P}(N(k)\ge n)=\mathbb{P}(S_n\le_d k)=F_n(k),   
\end{equation}

which gives the survival function.
Since $N(k)$ is integer-valued, $\{N(k)\le n\}$ is the complement of $\{N(k)\ge n+1\}$, hence
\[
F_{N(k)}(n)=1-\overline{F}_{N(k)}(n+1)=1-F_{n+1}(k).
\]
Finally,
\[
f_{N(k)}(n)=\mathbb{P}(N(k)\ge n)-\mathbb{P}(N(k)\ge n+1)=F_n(k)-F_{n+1}(k).
\]
\end{proof}

\subsection{Renewal equation for $M$ and solution via convolution inverse}\label{subsec:renewal-equation-solution}

\noindent The renewal function satisfies the following multi-time renewal equation.
\begin{proposition}\label{prop:renewal-eq-M}
The renewal function $M$ satisfies the discrete multi-time renewal equation
\begin{equation}\label{eq:renewal-eq-M}
M = F + f*M,
\end{equation}
i.e.\ for all $k\in\mathbb{N}^d$,
\[
M(k)=F(k)+\sum_{\ell\le_d k} f(\ell)\,M(k-\ell).
\]
\end{proposition}

\begin{proof}
Condition on the first increment $X_1$.
If $X_1\not\le_d k$, then $N(k)=0$.
If $X_1\le_d k$, then after the first renewal the remaining number of renewals inside the shifted box is distributed as
$N(k-X_1)$ (by independence and stationarity of increments), hence
\[
N(k)=\mathbf{1}_{\{X_1\le_d k\}}\big(1+N'(k-X_1)\big),
\]
where $N'$ is an independent copy of $N$.
Taking expectations gives
\[
M(k)=\mathbb{P}(X_1\le_d k)+\sum_{\ell\le_d k} f(\ell)\,M(k-\ell)
=F(k)+(f*M)(k).
\]
\end{proof}

\noindent The preceding renewal equation gives the following explicit representation of the renewal function.
\begin{corollary}\label{cor:M-explicit}
Assume $f(0_d)<1$ so that $(\ez-f)(0_d)=1-f(0_d)\neq 0$ and $(\ez-f)$ is convolutionally invertible.
Then the renewal function admits the explicit representation
\begin{equation}\label{eq:renewal-function}
M=(\ez-f)^{(-1)}*F.
\end{equation}
If in addition $f(0_d)=0$, then $(\ez-f)^{(-1)}=\sum_{n\ge 0} f^{(n)}$ coefficientwise and
\[
M(k)=\sum_{n\ge 1} F_n(k)=\sum_{n\ge 1}\mathbb{P}(S_n\le_d k),
\qquad k\in\mathbb{N}^d,
\]
where the sum is finite for each fixed $k$.
\end{corollary}

\begin{proof}
From \eqref{eq:renewal-eq-M}, we have $(\ez-f)*M=F$.
Since $(\ez-f)$ is invertible, multiplying both sides by $(\ez-f)^{(-1)}$ yields \eqref{eq:renewal-function}.
If $f(0_d)=0$, then $f^{(n)}(k)=0$ for all $n>|k|_1$, so the geometric expansion 
\begin{equation*}
(\ez-f)^{(-1)}=\sum_{n\ge0} f^{(n)}
\end{equation*}
holds coefficientwise, and combining with \eqref{eq:F-n-conv} gives the stated series representation for $M$.
\end{proof}

\subsection{A general multi-time renewal equation}\label{subsec:general-renewal-equation}

Many multi-time quantities satisfy convolution fixed-point equations on $\mathbb{N}^d$.
Throughout this subsection, $*$ denotes the multi-index convolution on the relevant lattice
($\mathbb{N}^d$, $\mathbb{N}^m$, or $\mathbb{N}$), and $\ez$ and $\sd$ denote the corresponding identity
and constant-one sequences (the underlying dimension will always be clear from context).

\noindent We use the following form for a general discrete multi-time renewal equation.
\begin{definition}\label{def:general-renewal}
Let $f:\mathbb{N}^d\to[0,1]$ be a pmf and let $b:\mathbb{N}^d\to\mathbb{R}$ be given.
A (discrete) multi-time renewal equation is of the form
\begin{equation}\label{eq:renewal-equation}
g=b+f*g,
\end{equation}
where $g:\mathbb{N}^d\to\mathbb{R}$ is the unknown.
\end{definition}

\noindent The next theorem gives existence and uniqueness of the solution.
\begin{theorem}\label{thm:general-renewal-solution}
If $f(0_d)<1$, then $(\ez-f)$ is convolutionally invertible and \eqref{eq:renewal-equation}
has a unique solution given by
\begin{equation}\label{eq:general-solution}
g=(\ez-f)^{(-1)}*b.
\end{equation}
\end{theorem}

\begin{proof}
Equation \eqref{eq:renewal-equation} is equivalent to $(\ez-f)*g=b$.
Since $(\ez-f)(0_d)=1-f(0_d)\neq 0$, the convolutional inverse exists and multiplying by
$(\ez-f)^{(-1)}$ yields \eqref{eq:general-solution}.
Uniqueness follows similarly: if $g_1,g_2$ solve \eqref{eq:renewal-equation}, then
$(\ez-f)*(g_1-g_2)=0$ and hence $g_1=g_2$.
\end{proof}

\noindent Instantaneous renewals may be removed by the embedded construction described below.
\begin{remark}
\label{rem:no-batch}
Assume $p_0:=f(0_d)\in[0,1)$ and set $q_0:=1-p_0$.
If $p_0>0$, the increment distribution has an atom at $0_d$, which corresponds to
\emph{instantaneous renewals} (renewals that do not advance any time coordinate).
A standard way to exclude such instantaneous events is to work with the \emph{embedded}
increment distribution obtained by conditioning away the atom at $0_d$.

Define $f^{*}:\mathbb{N}^d\to[0,1]$ by
\begin{equation}\label{eq:f-star-def}
f^{*}(0_d):=0,
\qquad
f^{*}(k):=\frac{f(k)}{q_0}\ \ \text{for }k\neq 0_d.
\end{equation}
Equivalently, $X^{*}\stackrel{d}{=}(X\mid X\neq 0_d)$ and
\begin{equation}\label{eq:f-decompose}
f = p_0\,\ez + q_0\,f^{*}.
\end{equation}
Operationally, the embedded (batch-free) renewal chain is obtained by deleting all indices
$n$ with $X_n=0_d$; the remaining increments are i.i.d.\ with pmf $f^{*}$.

\smallskip
\noindent\emph{Effect on the renewal kernel and on renewal equations.}
From \eqref{eq:f-decompose} we have
\begin{equation}\label{eq:kernel-scale}
\ez-f = q_0(\ez-f^{*})
\quad\Longrightarrow\quad
(\ez-f)^{(-1)}=\frac{1}{q_0}\,(\ez-f^{*})^{(-1)}.
\end{equation}
Consequently, any renewal equation $g=b+f*g$ can be rewritten as
\begin{equation}\label{eq:renewal-eq-embedded}
g=\frac{1}{q_0}\,b + f^{*}*g,
\end{equation}
and therefore
\begin{equation}\label{eq:g-embedded-solution}
g=(\ez-f)^{(-1)}*b
=\frac{1}{q_0}\,(\ez-f^{*})^{(-1)}*b.
\end{equation}

\smallskip
\noindent\emph{Effect on the renewal function.}
Recall that the multi-time cdf of $X$ is $F:=\sd*f$ and the renewal function is
$M:=\mathbb{E}[N]=(\ez-f)^{(-1)}*F$. Define similarly
$F^{*}:=\sd*f^{*}$ and
\begin{equation}\label{eq:M-star-def}
M^{*}:=(\ez-f^{*})^{(-1)}*F^{*},
\end{equation}
the renewal function of the embedded chain.
From \eqref{eq:f-decompose} we also have
\begin{equation}\label{eq:F-decompose}
F = p_0\,\sd + q_0\,F^{*}.
\end{equation}
Combining \eqref{eq:kernel-scale}--\eqref{eq:F-decompose} yields the practical conversion
formula
\begin{equation}\label{eq:M-conversion}
M
=\frac{1}{q_0}\,M^{*}+\frac{p_0}{q_0}\,\sd.
\end{equation}
In particular, $M$ is computable from the batch-free model $(f^{*},F^{*},M^{*})$ and the
single scalar $p_0=f(0_d)$.

\smallskip
\noindent\emph{Finite-sum expansion when $f^{*}(0_d)=0$.}
Since $f^{*}(0_d)=0$, the convolutional inverse $(\ez-f^{*})^{(-1)}$ admits the
coefficientwise geometric expansion
$(\ez-f^{*})^{(-1)}=\sum_{n\ge 0}(f^{*})^{(n)}$, and all standard finite-sum identities
from the no-instantaneous case apply to the embedded chain.
\smallskip
\noindent
\emph{Neumann-series representation.}
More generally, let $\alpha\in\mathbb{R}$ and assume $|\alpha|\,f(0_d)<1$.
Then the coefficientwise series
\begin{equation*}
\psi_\alpha := \sum_{n\ge 0} \alpha^n f^{(n)}
\end{equation*}
is well defined (each coefficient $\psi_\alpha(k)$ is finite) and satisfies
\begin{equation*}
(\ez-\alpha f)*\psi_\alpha=\ez,
\end{equation*}
hence $\psi_\alpha=(\ez-\alpha f)^{(-1)}$.
In particular, for $\alpha=1$ one recovers
$(\ez-f)^{(-1)}=\sum_{n\ge 0} f^{(n)}$
even when $f(0_d)>0$.
When $f(0_d)=0$, the series truncates at $n\le |k|_1$ for each fixed $k$ by Lemma~\ref{lem:vanishing-powers}.
\end{remark}

\noindent Coordinate projections lead to the following marginal renewal chains.
\begin{definition}\label{def:marginal-renewal}
Let $U\subseteq\{1,\ldots,d\}$ be a nonempty index set with $|U|=m$ and let
$\pi_U:\mathbb{N}^d\to\mathbb{N}^m$ be the coordinate projection.
Define the projected increments $X_n^{[U]}:=\pi_U(X_n)$.
Then $(X_n^{[U]})_{n\ge 1}$ is an i.i.d.\ sequence on $\mathbb{N}^m$ with pmf
\begin{equation}\label{eq:f-marginal}
f^{[U]}(k_U):=\mathbb{P}(X_1^{[U]}=k_U),
\qquad k_U\in\mathbb{N}^m.
\end{equation}
Assume $f^{[U]}(0_m)<1$ (non-degeneracy in the projected coordinates), and define the
projected renewal chain and counting process by
\begin{equation}\label{eq:proj-chain}
S_0^{[U]}=0_m,
\qquad
S_n^{[U]}=\sum_{i=1}^n X_i^{[U]},
\qquad n\ge 1,
\end{equation}
and
\begin{equation}\label{eq:proj-count}
N^{[U]}(k_U):=\sup\{n\in\mathbb{N}: S_n^{[U]}\le_m k_U\},
\qquad k_U\in\mathbb{N}^m,
\end{equation}
where $\le_m$ denotes componentwise order in $\mathbb{N}^m$.
Its renewal function is
\begin{equation*}
M^{[U]}(k_U):=\mathbb{E}[N^{[U]}(k_U)].
\end{equation*}
\smallskip
\noindent
In particular, for $U=\{u\}$ one recovers the univariate marginal renewal process
$S_n^{[u]}=\sum_{i=1}^n X_i^{[u]}$ and the marginal counting process
$N^{[u]}(k_u)=\sup\{n:S_n^{[u]}\le k_u\}$.
\end{definition}

\noindent The renewal equation is preserved under coordinate projection.
\begin{remark}\label{rem:proj-stability}
For each nonempty $U$ as in Definition~\ref{def:marginal-renewal}, define the projected cdf
$F^{[U]}:=\sd*f^{[U]}$.
Then the same renewal equation and solution representation hold on $\mathbb{N}^m$:
\begin{equation}\label{eq:proj-renewal-eq}
M^{[U]} = F^{[U]} + f^{[U]}*M^{[U]},
\qquad
M^{[U]} = (\ez-f^{[U]})^{(-1)}*F^{[U]}.
\end{equation}
If $f^{[U]}(0_m)>0$, one may remove instantaneous renewals in the \emph{projected} model
by repeating the embedded construction of Remark~\ref{rem:no-batch}, i.e.\ by defining
$f^{[U],*}(0_m)=0$ and $f^{[U],*}(k_U)=f^{[U]}(k_U)/(1-f^{[U]}(0_m))$ for $k_U\neq 0_m$.
\end{remark}

\noindent The second mixed moment of two marginal counts can be expressed through a bivariate renewal function.
\begin{proposition}
\label{prop:cov-marginals}
Fix two distinct coordinates $u\neq v$ and times $k_u,k_v\in\mathbb{N}$.
Let $M^{[u]}(k_u)=\mathbb{E}[N^{[u]}(k_u)]$ and $M^{[v]}(k_v)=\mathbb{E}[N^{[v]}(k_v)]$.
Let $M^{[u,v]}:\mathbb{N}^2\to\mathbb{R}$ denote the \emph{bivariate} renewal function of
the projected increment vector $(X^{[u]},X^{[v]})$, i.e.
$M^{[u,v]}(k_u,k_v)=\mathbb{E}[N^{[u,v]}(k_u,k_v)]$ with
$N^{[u,v]}(k_u,k_v)=\sup\{n:(S_n^{[u]},S_n^{[v]})\le_2(k_u,k_v)\}$.

Define the univariate renewal kernels
$\psi_u:=(\ez-f^{[u]})^{(-1)}$ and $\psi_v:=(\ez-f^{[v]})^{(-1)}$ on $\mathbb{N}$, and
define $\widetilde{\psi}_{u,v}:\mathbb{N}^2\to\mathbb{R}$ by
\begin{equation}\label{eq:psi-tilde-uv}
\widetilde{\psi}_{u,v}(r_u,r_v)
=
\left\{
\begin{array}{ll}
\psi_u(r_u), & \text{if } r_u>0 \text{ and } r_v=0,\\[0.2em]
\psi_v(r_v), & \text{if } r_v>0 \text{ and } r_u=0,\\[0.2em]
\psi_u(0)+\psi_v(0)-1, & \text{if } r_u=r_v=0,\\[0.2em]
0, & \text{otherwise.}
\end{array}
\right.
\end{equation}
Then
\begin{equation}\label{eq:second-moment-uv}
\mathbb{E}\!\left[N^{[u]}(k_u)\,N^{[v]}(k_v)\right]
=
\big[M^{[u,v]}*\widetilde{\psi}_{u,v}\big](k_u,k_v),
\end{equation}
and consequently
\begin{equation}\label{eq:cov-uv}
\mathrm{Cov}\!\left(N^{[u]}(k_u),\,N^{[v]}(k_v)\right)
=
\big[M^{[u,v]}*\widetilde{\psi}_{u,v}\big](k_u,k_v)
-
M^{[u]}(k_u)\,M^{[v]}(k_v).
\end{equation}
\end{proposition}

\begin{proof}
Using the tail-sum identity for nonnegative integer-valued random variables,
\begin{equation}\label{eq:tail-sum-product}
\mathbb{E}\!\left[N^{[u]}(k_u)\,N^{[v]}(k_v)\right]
=
\sum_{n_u\ge 1}\sum_{n_v\ge 1}
\mathbb{P}\!\left(N^{[u]}(k_u)\ge n_u,\ N^{[v]}(k_v)\ge n_v\right).
\end{equation}
By the basic equivalence $N^{[u]}(k_u)\ge n_u \Longleftrightarrow S_{n_u}^{[u]}\le k_u$,
we can rewrite each summand in \eqref{eq:tail-sum-product} as
$\mathbb{P}(S_{n_u}^{[u]}\le k_u,\ S_{n_v}^{[v]}\le k_v)$.

Split the double sum into the three cases $n_u=n_v$, $n_u<n_v$, and $n_v<n_u$.
The diagonal part is exactly the bivariate renewal function:
\begin{equation}\label{eq:diag-part}
\sum_{n\ge 1}\mathbb{P}(S_n^{[u]}\le k_u,\ S_n^{[v]}\le k_v)
=
M^{[u,v]}(k_u,k_v).
\end{equation}
Consider next $n_u<n_v$ and write $n_v=n_u+m$ with $m\ge 1$.
By independence of increments after time $n_u$, conditioning on
$R_v:=S_{n_u+m}^{[v]}-S_{n_u}^{[v]}$ yields
\begin{equation}\label{eq:cond-step}
\mathbb{P}(S_{n_u}^{[u]}\le k_u,\ S_{n_u+m}^{[v]}\le k_v)
=
\sum_{r_v\le k_v}
\mathbb{P}(S_{n_u}^{[u]}\le k_u,\ S_{n_u}^{[v]}\le k_v-r_v)\,
f^{[v],(m)}(r_v),
\end{equation}
where $f^{[v],(m)}$ is the $m$-fold convolution power of the marginal pmf $f^{[v]}$.
Summing \eqref{eq:cond-step} over $m\ge 1$ gives the factor
$\sum_{m\ge 1} f^{[v],(m)}(r_v)=\psi_v(r_v)-\ez(r_v)$, which equals $\psi_v(r_v)$ for
$r_v>0$ and equals $\psi_v(0)-1$ for $r_v=0$.
Summing also over $n_u\ge 1$ identifies the remaining term as $M^{[u,v]}(k_u,k_v-r_v)$.

By symmetry, the contribution of the region $n_v<n_u$ produces the analogous term
with $\psi_u$ along the other axis. Collecting the diagonal term \eqref{eq:diag-part}
together with the two off-diagonal contributions yields precisely the convolution form
\eqref{eq:second-moment-uv} with the axis-supported kernel \eqref{eq:psi-tilde-uv}.
Finally, \eqref{eq:cov-uv} follows from the identity
$\mathrm{Cov}(X,Y)=\mathbb{E}[XY]-\mathbb{E}[X]\mathbb{E}[Y]$.
\end{proof}

\noindent The generating-function form of the renewal equation is the following.
\begin{remark}\label{rem:pgf-renewal}
Using the ring isomorphism $a\mapsto P_a$ (Proposition~\ref{prop:ring}),
the renewal equation \eqref{eq:renewal-equation} is equivalent to
\begin{equation}\label{eq:pgf-renewal-identity}
P_g(x_{1:d})
=
P_b(x_{1:d}) + P_f(x_{1:d})\,P_g(x_{1:d})
\qquad\text{in }\mathbb{R}[[x_1,\ldots,x_d]].
\end{equation}
Since the constant term of $P_f$ is $P_f(0,\ldots,0)=f(0_d)$, the series
$1-P_f$ has nonzero constant term if and only if $f(0_d)<1$.
In that case it admits a unique reciprocal $(1-P_f)^{-1}\in\mathbb{R}[[x_1,\ldots,x_d]]$,
and we obtain
\begin{equation}\label{eq:pgf-renewal-solution}
P_g(x_{1:d})
=
\big(1-P_f(x_{1:d})\big)^{-1} P_b(x_{1:d})
=
\frac{P_b(x_{1:d})}{1-P_f(x_{1:d})}.
\end{equation}
(If one evaluates these formal series at $|x_u|\le 1$, then $P_f$ coincides with the
multivariate probability generating function of $X_1$.)
\end{remark}

The representation \eqref{eq:general-solution} will be used repeatedly in the sequel
to express reliability and availability functionals as solutions of multi-time renewal
equations, and to develop asymptotic results for renewal counts and additive
functionals.

\section{Limit theorems}\label{sec:limit}

In this section we establish limit results for the multi-time renewal chain
$S_n=\sum_{m=1}^n X_m$ and its associated counting process
\begin{equation}\label{eq:N-limit-def}
N(k)=\sup\{n\in\mathbb{N}: S_n\le_d k\},
\qquad k\in\mathbb{N}^d,
\end{equation}
as defined in subsection~\ref{subsec:renewal-chain-counting}.
We work under the standing assumptions that the increments
$X_1,X_2,\ldots$ are i.i.d.\ $\mathbb{N}^d$-valued random vectors with
\begin{equation}\label{eq:mu-u-def}
\mu_u:=\mathbb{E}\!\left[X_1^{[u]}\right]\in(0,\infty),
\qquad u=1,\ldots,d.
\end{equation}
Whenever needed for second-order limits, we also assume
$\mathbb{E}\!\left[(X_1^{[u]})^2\right]<\infty$ for the relevant coordinates.

For $k=(k_1,\ldots,k_d)\in\mathbb{N}^d$, we set $|k|_1:=k_1+\cdots+k_d$.
We study asymptotics along proportional growth directions as follows.
Fix $\lambda=(\lambda^1,\ldots,\lambda^d)$ with $\lambda^u\in(0,1)$ and
$\sum_{u=1}^d \lambda^u=1$, and write
\begin{equation}\label{eq:lambda-scaling}
k\stackrel{\lambda}{\to}\infty
\quad\Longleftrightarrow\quad
|k|_1\to\infty
\ \ \text{and}\ \ 
\frac{k_u}{|k|_1}\to \lambda^u,\quad u=1,\ldots,d.
\end{equation}
(Under \eqref{eq:lambda-scaling}, each coordinate $k_u\to\infty$ since $\lambda^u>0$.)
Define the directional renewal rate
\begin{equation}\label{eq:mu-lambda}
\mu_\lambda:=\min_{1\le u\le d}\left\{\frac{\lambda^u}{\mu_u}\right\}.
\end{equation}

\subsection{Convergence results}\label{subsec:limit-conv}

We begin with a basic random-index lemma (a multidimensional version of a standard
subsequence principle). As usual in multi-index asymptotics, limits are understood
along deterministic sequences $k^{(m)}\in\mathbb{N}^d$ with $k^{(m)}_u\to\infty$.

\begin{lemma}\label{lem:random-index}
Let $(Y_n)_{n\ge 1}$ be random variables with
$Y_n\xrightarrow[n\to\infty]{a.s.} Y$.
Let $(T(k))_{k\in\mathbb{N}^d}$ be integer-valued random variables such that, for every
deterministic sequence $k^{(m)}\in\mathbb{N}^d$ with $k^{(m)}_u\to\infty$ for all $u$,
\begin{equation*}
T\!\left(k^{(m)}\right)\xrightarrow[m\to\infty]{a.s.}\infty.
\end{equation*}
Then, along the same class of sequences $k^{(m)}$,
\begin{equation*}
Y_{T(k^{(m)})}\xrightarrow[m\to\infty]{a.s.} Y.
\end{equation*}
\end{lemma}

\begin{lemma}\label{lem:Sn-N-diverge}
For each coordinate $u=1,\ldots,d$,
\begin{equation*}
\frac{S_n^{[u]}}{n}\xrightarrow[n\to\infty]{a.s.}\mu_u,
\qquad\text{hence}\qquad
S_n^{[u]}\xrightarrow[n\to\infty]{a.s.}\infty.
\end{equation*}
Moreover, for any deterministic sequence $k^{(m)}\in\mathbb{N}^d$ such that
$k^{(m)}_u\to\infty$ for all $u$, one has
\begin{equation*}
N\!\left(k^{(m)}\right)\xrightarrow[m\to\infty]{a.s.}\infty.
\end{equation*}
\end{lemma}

\begin{proof}
The coordinatewise strong law of large numbers gives
$S_n^{[u]}/n\to\mu_u$ almost surely for each $u$, and since $\mu_u>0$ this implies
$S_n^{[u]}\to\infty$ almost surely.

Let $(k^{(m)})$ be deterministic with $k^{(m)}_u\to\infty$ for all $u$.
Fix $\omega$ in the almost sure event where all coordinatewise SLLNs hold.
Then for each fixed $n\in\mathbb{N}$, the vector $S_n(\omega)$ is finite, and since
$k^{(m)}\to\infty$ componentwise there exists an index $m_n(\omega)$ such that
$S_n(\omega)\le_d k^{(m)}$ for all $m\ge m_n(\omega)$.
By the definition \eqref{eq:N-limit-def}, this implies
$N(k^{(m)})(\omega)\ge n$ for all $m\ge m_n(\omega)$.
Since $n$ is arbitrary, we conclude that $N(k^{(m)})(\omega)\to\infty$.
\end{proof}

Recall that the multi-time counting process satisfies
$N(k)=\min_{u} N^{[u]}(k_u)$, where
$N^{[u]}(k_u):=\sup\{n\in\mathbb{N}: S_n^{[u]}\le k_u\}$.
We will use the classical univariate renewal SLLN (which holds under $\mu_u\in(0,\infty)$,
even when $\mathbb{P}(X_1^{[u]}=0)>0$):
\begin{equation}\label{eq:univariate-SLLN}
\frac{N^{[u]}(k_u)}{k_u}\xrightarrow[k_u\to\infty]{a.s.}\frac{1}{\mu_u},
\qquad u=1,\ldots,d.
\end{equation}

We now give the directional strong law for $N$.
\begin{proposition}\label{prop:SLLN-N}
Under \eqref{eq:lambda-scaling},
\begin{equation}\label{eq:SLLN-N}
\frac{N(k)}{|k|_1}\xrightarrow[k\stackrel{\lambda}{\to}\infty]{a.s.}\mu_\lambda.
\end{equation}
\end{proposition}

\begin{proof}
Using $N(k)=\min_{1\le u\le d}N^{[u]}(k_u)$, we have
\begin{equation*}
\frac{N(k)}{|k|_1}
=
\min_{1\le u\le d}\left\{\frac{N^{[u]}(k_u)}{|k|_1}\right\}
=
\min_{1\le u\le d}\left\{\frac{N^{[u]}(k_u)}{k_u}\cdot \frac{k_u}{|k|_1}\right\}.
\end{equation*}
Under \eqref{eq:lambda-scaling}, for each fixed $u$ we have
$k_u\to\infty$ and $k_u/|k|_1\to\lambda^u$, and by \eqref{eq:univariate-SLLN} we obtain
almost surely
\begin{equation*}
\frac{N^{[u]}(k_u)}{k_u}\cdot \frac{k_u}{|k|_1}
\xrightarrow[k\stackrel{\lambda}{\to}\infty]{a.s.}
\frac{1}{\mu_u}\cdot \lambda^u.
\end{equation*}
Taking the minimum over $u=1,\ldots,d$ (a finite minimum) yields \eqref{eq:SLLN-N}.
\end{proof}

Let $M(k)=\mathbb{E}[N(k)]$ be the multi-time renewal function.
To transfer the almost sure limit to expectations, it is convenient to separate the
case $f(0_d)=0$ (direct domination) from $f(0_d)>0$ (reduction to the embedded chain
of Remark~\ref{rem:no-batch}).

\begin{proposition}\label{prop:SLLN-M}
Assume \eqref{eq:mu-u-def} and $f(0_d)<1$. Under \eqref{eq:lambda-scaling},
\begin{equation}\label{eq:SLLN-M}
\frac{M(k)}{|k|_1}\xrightarrow[k\stackrel{\lambda}{\to}\infty]{}\mu_\lambda.
\end{equation}
\end{proposition}

\begin{proof}
\emph{Case 1: $f(0_d)=0$.}
By Proposition~\ref{prop:SLLN-N}, we have
$N(k)/|k|_1\to \mu_\lambda$ almost surely along $k\stackrel{\lambda}{\to}\infty$.
Under $f(0_d)=0$ we have $|X_n|_1\ge 1$ almost surely, hence
$|S_n|_1\ge n$ and therefore $0\le N(k)\le |k|_1$ for all $k$.
Thus $0\le N(k)/|k|_1\le 1$, and the dominated convergence theorem yields
\begin{equation*}
\frac{M(k)}{|k|_1}
=
\mathbb{E}\!\left[\frac{N(k)}{|k|_1}\right]
\xrightarrow[k\stackrel{\lambda}{\to}\infty]{}
\mu_\lambda.
\end{equation*}

\smallskip
\emph{Case 2: $f(0_d)>0$.}
Let $p_0:=f(0_d)$ and $q_0:=1-p_0$, and define the embedded (batch-free) increment
pmf $f^{*}$ as in \eqref{eq:f-star-def}. Let $M^{*}$ be the renewal function of the
embedded chain, as in \eqref{eq:M-star-def}.
By the conversion formula \eqref{eq:M-conversion} (Remark~\ref{rem:no-batch}),
\begin{equation*}
M(k)=\frac{1}{q_0}\,M^{*}(k)+\frac{p_0}{q_0}\,\sd(k).
\end{equation*}
Since $f^{*}(0_d)=0$, Case~1 applies to the embedded chain, thus
$M^{*}(k)/|k|_1\to \mu_\lambda^{*}$, where
$\mu_\lambda^{*}=\min_{u} \lambda^u/\mu_u^{*}$ and
$\mu_u^{*}=\mathbb{E}[X_1^{*[u]}]=\mu_u/q_0$. Hence $\mu_\lambda^{*}=q_0\,\mu_\lambda$.
Dividing the conversion identity by $|k|_1$ and letting $k\stackrel{\lambda}{\to}\infty$ gives
\begin{equation*}
\frac{M(k)}{|k|_1}
=
\frac{1}{q_0}\,\frac{M^{*}(k)}{|k|_1}
+
\frac{p_0}{q_0}\,\frac{\sd(k)}{|k|_1}
\longrightarrow
\frac{1}{q_0}\,(q_0\mu_\lambda)+0
=
\mu_\lambda,
\end{equation*}
which proves \eqref{eq:SLLN-M}.
\end{proof}

\subsection{Central limit theorem for additive functionals}\label{subsec:clt-additive}

Let $s\in\mathbb{N}$ and let $h:\mathbb{N}^d\to\mathbb{R}^s$ be measurable with
$\mathbbm{E}\!\left[h(X_1)^{\top}h(X_1)\right]<\infty$.
Define the additive functional
\begin{equation}\label{eq:Wh-def}
W_h(k_{1:d})
:=
\sum_{n=1}^{N(k_{1:d})} h(X_n),
\qquad
k_{1:d}\in\mathbb{N}^d,
\end{equation}
with the convention $W_h(k_{1:d})=0_s$ when $N(k_{1:d})=0$.

\noindent Set
\begin{equation}\label{eq:Ah-Bh}
A_h:=\mathbbm{E}[h(X_1)]\in\mathbb{R}^s,
\ \
B_h:=\mathbbm{E}\!\left[h(X_1)\,h(X_1)^{\top}\right]\in\mathbb{R}^{s\times s},
\ \
\Sigma_h:=B_h-A_hA_h^{\top}.
\end{equation}
Then $\Sigma_h=\mathrm{Cov}(h(X_1))$ is symmetric and positive semidefinite. The standard i.i.d. multivariate CLT yields
\begin{equation}\label{eq:clt-iid-h}
\frac{1}{\sqrt{n}}\sum_{m=1}^n\big(h(X_m)-A_h\big)
\xrightarrow[n\to\infty]{d}
\mathcal{N}_s(0,\Sigma_h),
\end{equation}
where the Gaussian limit is understood in the possibly degenerate sense when $\Sigma_h$ is singular.

\smallskip
The next lemma is an Anscombe-type result adapted to our directional scaling.

\noindent The following Anscombe-type lemma is used for random renewal indices.
\begin{lemma}\label{lem:anscombe-multitime}
Assume the hypotheses of Proposition~\ref{prop:SLLN-N}, so that under \eqref{eq:lambda-scaling} one has
$N(k_{1:d})/|k|_1\to\mu_\lambda$ almost surely.
Let $(Y_n)_{n\ge 1}$ be i.i.d.\ real-valued random variables with
$\mathbbm{E}[Y_1]=0$ and $\mathbbm{E}[Y_1^2]=1$, and set $T_n=\sum_{m=1}^n Y_m$.
Then, under \eqref{eq:lambda-scaling},
\begin{equation}\label{eq:anscombe}
\frac{T_{N(k_{1:d})}}{\sqrt{|k|_1}}
\xrightarrow[k_{1:d}\stackrel{\lambda}{\to}\infty]{d}
\mathcal{N}(0,\mu_\lambda).
\end{equation}
\end{lemma}

\begin{proof}
Let $z_k:=\lfloor \mu_\lambda |k|_1\rfloor$.
By Proposition~\ref{prop:SLLN-N}, $N(k_{1:d})/|k|_1\to\mu_\lambda$ almost surely, and since
$z_k/|k|_1\to\mu_\lambda$ deterministically, it follows that
\begin{equation}\label{eq:N-over-z}
\frac{N(k_{1:d})}{z_k}\xrightarrow[k_{1:d}\stackrel{\lambda}{\to}\infty]{a.s.}1.
\end{equation}
Fix $\varepsilon\in(0,1/3)$ and define
\begin{equation*}
z_1:=\lfloor (1-\varepsilon^3)z_k\rfloor+1,
\qquad
z_2:=\lfloor (1+\varepsilon^3)z_k\rfloor.
\end{equation*}
By \eqref{eq:N-over-z},
\begin{equation}\label{eq:prob-band}
\mathbbm{P}\!\left(N(k_{1:d})\notin[z_1,z_2]\right)\to 0.
\end{equation}
Moreover,
\begin{equation*}
\mathbbm{P}\!\left(\big|T_{N(k_{1:d})}-T_{z_k}\big|>\varepsilon \sqrt{z_k}\right)
\le I_1(k)+I_2(k),
\end{equation*}
where
\begin{eqnarray*}
I_1(k)&:=&\mathbbm{P}\!\left(\big|T_{N(k_{1:d})}-T_{z_k}\big|>\varepsilon \sqrt{z_k},\ N(k_{1:d})\in[z_1,z_2]\right),\\
I_2(k)&:=&\mathbbm{P}\!\left(N(k_{1:d})\notin[z_1,z_2]\right).
\end{eqnarray*}
Clearly $I_2(k)\to 0$ by \eqref{eq:prob-band}. On the event $\{N(k_{1:d})\in[z_1,z_2]\}$,
\begin{equation*}
\big|T_{N(k_{1:d})}-T_{z_k}\big|
\le \max_{z_1\le n\le z_k}\big|T_n-T_{z_k}\big|
+\max_{z_k\le n\le z_2}\big|T_n-T_{z_k}\big|.
\end{equation*}
Applying Kolmogorov's maximal inequality to the partial sums of $(Y_n)$ yields
\begin{eqnarray*}
\mathbbm{P}\!\left(\max_{z_1\le n\le z_k}\big|T_n-T_{z_k}\big|>\varepsilon\sqrt{z_k}\right)
&\le& \frac{z_k-z_1}{\varepsilon^2 z_k},
\\
\mathbbm{P}\!\left(\max_{z_k\le n\le z_2}\big|T_n-T_{z_k}\big|>\varepsilon\sqrt{z_k}\right)
&\le& \frac{z_2-z_k}{\varepsilon^2 z_k}.
\end{eqnarray*}
Since $z_k-z_1\le \varepsilon^3 z_k+1$ and $z_2-z_k\le \varepsilon^3 z_k$, we obtain
$I_1(k)\le 2\varepsilon+o(1)$. As $\varepsilon>0$ is arbitrary,
\begin{equation}\label{eq:T-diff-op}
\frac{T_{N(k_{1:d})}-T_{z_k}}{\sqrt{z_k}}
\xrightarrow[k_{1:d}\stackrel{\lambda}{\to}\infty]{p}0.
\end{equation}
By the classical CLT, $T_{z_k}/\sqrt{z_k}\xrightarrow[]{d} \mathcal{N}(0,1)$, and since
$\sqrt{z_k/|k|_1}\to \sqrt{\mu_\lambda}$, Slutsky's lemma gives
\begin{equation*}
\frac{T_{z_k}}{\sqrt{|k|_1}}
=\sqrt{\frac{z_k}{|k|_1}}\cdot \frac{T_{z_k}}{\sqrt{z_k}}
\xrightarrow[k_{1:d}\stackrel{\lambda}{\to}\infty]{d}
\mathcal{N}(0,\mu_\lambda).
\end{equation*}
Finally, \eqref{eq:T-diff-op} implies
$(T_{N(k_{1:d})}-T_{z_k})/\sqrt{|k|_1}\xrightarrow[]{p}0$, hence
$T_{N(k_{1:d})}/\sqrt{|k|_1}$ has the same limit, proving \eqref{eq:anscombe}.
\end{proof}

We now state the CLT for additive functionals.

\noindent We now obtain the central limit theorem for additive functionals.
\begin{theorem}\label{thm:clt-additive}
Assume $\mathbbm{E}\!\left[h(X_1)^{\top}h(X_1)\right]<\infty$.
Under \eqref{eq:lambda-scaling},
\begin{equation}\label{eq:clt-additive}
\frac{W_h(k_{1:d})-A_h\,N(k_{1:d})}{\sqrt{|k|_1}}
\xrightarrow[k_{1:d}\stackrel{\lambda}{\to}\infty]{d}
\mathcal{N}_s\!\left(0,\mu_\lambda\,\Sigma_h\right),
\end{equation}
where $A_h$ and $\Sigma_h$ are given by \eqref{eq:Ah-Bh}.
\end{theorem}

\begin{proof}
Let $a\in\mathbb{R}^s$ be arbitrary and consider the scalar projection
\begin{equation*}
Z_n^{(a)}:=a^{\top}\big(h(X_n)-A_h\big),
\qquad n\ge 1.
\end{equation*}
Then $(Z_n^{(a)})$ are i.i.d.\ with $\mathbbm{E}[Z_1^{(a)}]=0$ and
$\mathrm{Var}(Z_1^{(a)})=a^{\top}\Sigma_h a$.
If $a^{\top}\Sigma_h a=0$, then $Z_1^{(a)}=0$ almost surely and thus
$a^{\top}\big(W_h(k_{1:d})-A_hN(k_{1:d})\big)=0$ almost surely for all $k_{1:d}$, so the limit in
\eqref{eq:clt-additive} holds trivially along this direction.

Assume now that $a^{\top}\Sigma_h a>0$ and define the standardized i.i.d.\ sequence
\begin{equation*}
Y_n^{(a)}:=\frac{Z_n^{(a)}}{\sqrt{a^{\top}\Sigma_h a}},
\qquad
T_n^{(a)}:=\sum_{m=1}^n Y_m^{(a)}.
\end{equation*}
Then $\mathbbm{E}[Y_1^{(a)}]=0$ and $\mathbbm{E}[(Y_1^{(a)})^2]=1$. Moreover,
\begin{equation*}
\frac{a^{\top}\big(W_h(k_{1:d})-A_hN(k_{1:d})\big)}{\sqrt{|k|_1}}
=
\sqrt{a^{\top}\Sigma_h a}\cdot
\frac{T_{N(k_{1:d})}^{(a)}}{\sqrt{|k|_1}}.
\end{equation*}
By Lemma~\ref{lem:anscombe-multitime},
$T_{N(k_{1:d})}^{(a)}/\sqrt{|k|_1}\xrightarrow[]{d}\mathcal{N}(0,\mu_\lambda)$, hence
\begin{equation*}
\frac{a^{\top}\big(W_h(k_{1:d})-A_hN(k_{1:d})\big)}{\sqrt{|k|_1}}
\xrightarrow[k_{1:d}\stackrel{\lambda}{\to}\infty]{d}
\mathcal{N}\!\left(0,\mu_\lambda\,a^{\top}\Sigma_h a\right).
\end{equation*}
Since this holds for every $a\in\mathbb{R}^s$, the Cram\'er--Wold device yields the
vector convergence \eqref{eq:clt-additive}.
\end{proof}

\paragraph{Asymptotic normality of $N$ under a unique minimizer.}
A CLT for $N(k_{1:d})$ itself requires additional structure, because
$N(k_{1:d})=\min_{1\le u\le d}N^{[u]}(k_u)$ is a minimum of strongly dependent marginal counting
processes. A clean Gaussian limit holds when the minimum in \eqref{eq:mu-lambda} is attained at a
unique coordinate.

\noindent A Gaussian limit for the counting process itself holds under a unique minimizing coordinate.
\begin{corollary}\label{cor:clt-N-unique}
Assume $\mathbbm{E}\!\left[(X_1^{[u]})^2\right]<\infty$ for all $u=1,\ldots,d$ and let
$k_{1:d}\stackrel{\lambda}{\to}\infty$ in the sense of \eqref{eq:lambda-scaling}.
Set $\mu_u:=\mathbbm{E}[X_1^{[u]}]$ and
\begin{equation*}
\mu_\lambda:=\min_{1\le u\le d}\left\{\frac{\lambda^u}{\mu_u}\right\}.
\end{equation*}
\noindent
\emph{Interpretation.}
Since $N(k)=\min_{u}N^{[u]}(k_u)$ and $N^{[u]}(k_u)\approx k_u/\mu_u$ for large $k_u$,
the smallest linear trend $\min_u k_u/\mu_u$ governs the asymptotic growth of $N(k)$.
If the minimizer is unique, the corresponding coordinate becomes the unique active constraint
and eventually $N(k)=N^{[u_\lambda]}(k_{u_\lambda})$ along the direction $\lambda$.

Now, assume that this minimum is attained at a unique index $u_\lambda\in\{1,\ldots,d\}$, i.e.
\begin{equation}\label{eq:unique-min}
\frac{\lambda^{u_\lambda}}{\mu_{u_\lambda}}
<
\frac{\lambda^{v}}{\mu_{v}},
\qquad v\neq u_\lambda.
\end{equation}
Let $\sigma_{u_\lambda}^2:=\mathrm{Var}(X_1^{[u_\lambda]})$. Then
\begin{equation}\label{eq:clt-N-unique-centered-ku}
\frac{N(k_{1:d})-\frac{k_{u_\lambda}}{\mu_{u_\lambda}}}{\sqrt{|k|_1}}
\xrightarrow[k_{1:d}\stackrel{\lambda}{\to}\infty]{d}
\mathcal{N}\!\left(0,\ \lambda^{u_\lambda}\,\frac{\sigma_{u_\lambda}^2}{\mu_{u_\lambda}^3}\right).
\end{equation}
If in addition the deterministic scaling satisfies
$k_{u_\lambda}=\lambda^{u_\lambda}|k|_1+o(\sqrt{|k|_1})$, then \eqref{eq:clt-N-unique-centered-ku}
is equivalent to
\begin{equation}\label{eq:clt-N-unique}
\sqrt{|k|_1}\left(\frac{N(k_{1:d})}{|k|_1}-\mu_\lambda\right)
\xrightarrow[k_{1:d}\stackrel{\lambda}{\to}\infty]{d}
\mathcal{N}\!\left(0,\ \lambda^{u_\lambda}\,\frac{\sigma_{u_\lambda}^2}{\mu_{u_\lambda}^3}\right).
\end{equation}
\end{corollary}

\begin{proof}
By \eqref{eq:N-min-marginals} and \eqref{eq:univariate-SLLN}, for each $u$,
$N^{[u]}(k_u)/|k|_1\to \lambda^u/\mu_u$ almost surely under $k_{1:d}\stackrel{\lambda}{\to}\infty$.
On the probability-one event where all these limits hold simultaneously, the strict inequality
\eqref{eq:unique-min} implies that for all sufficiently large $k$ along the chosen direction,
\begin{equation*}
N(k_{1:d})=\min_{1\le u\le d} N^{[u]}(k_u)=N^{[u_\lambda]}(k_{u_\lambda}).
\end{equation*}
Therefore the asymptotics of $N(k_{1:d})$ reduce to the univariate renewal CLT for
$N^{[u_\lambda]}(k_{u_\lambda})$:
\begin{equation*}
\sqrt{k_{u_\lambda}}\left(\frac{N^{[u_\lambda]}(k_{u_\lambda})}{k_{u_\lambda}}-\frac{1}{\mu_{u_\lambda}}\right)
\xrightarrow[k_{u_\lambda}\to\infty]{d}
\mathcal{N}\!\left(0,\frac{\sigma_{u_\lambda}^2}{\mu_{u_\lambda}^3}\right).
\end{equation*}
Equivalently,
$N^{[u_\lambda]}(k_{u_\lambda})-\frac{k_{u_\lambda}}{\mu_{u_\lambda}}
=
\sqrt{k_{u_\lambda}}\cdot
\mathcal{N}\!\left(0,\frac{\sigma_{u_\lambda}^2}{\mu_{u_\lambda}^3}\right)
+o_p(\sqrt{k_{u_\lambda}})$.
Multiplying by $\sqrt{k_{u_\lambda}/|k|_1}\to\sqrt{\lambda^{u_\lambda}}$ and using Slutsky's lemma yields
\eqref{eq:clt-N-unique-centered-ku}.
Finally, if $k_{u_\lambda}=\lambda^{u_\lambda}|k|_1+o(\sqrt{|k|_1})$, then
$\left(k_{u_\lambda}-\lambda^{u_\lambda}|k|_1\right)/\sqrt{|k|_1}\to 0$, and since
$\mu_\lambda=\lambda^{u_\lambda}/\mu_{u_\lambda}$, the centering by $k_{u_\lambda}/\mu_{u_\lambda}$
can be replaced by $\mu_\lambda|k|_1$, giving \eqref{eq:clt-N-unique}.
\end{proof}

\begin{remark}\label{rem:nonunique}
If the minimum in \eqref{eq:mu-lambda} is not unique (e.g.\ in the ``balanced'' direction
$\lambda^u=\mu_u/(\mu_1+\cdots+\mu_d)$ for all $u$), then $N(k_{1:d})$ is asymptotically the minimum
of several competing marginal counting processes with comparable linear trends.
In that case, a Gaussian limit for $N(k_{1:d})$ is not automatic.
\smallskip
\noindent
\emph{Related bivariate asymptotics.}
In the two-dimensional renewal literature, joint weak limits for vectors of marginal
counting processes have been studied (see, for instance, Hunter's asymptotic results
in two-dimensional renewal theory \cite{hunter1974renewal1}).
Combined with the identity $N(k)=\min_u N^{[u]}(k_u)$, such joint Gaussian limits
naturally lead to non-Gaussian limits in the non-unique minimizer regime,
since the limit is then governed by the minimum of correlated Gaussian variables.
\end{remark}

\section{Maximum likelihood estimation}\label{sec:mle}

Let $m\in\mathbb{N}^d$ be a fixed observation horizon and consider a single trajectory of the
multi-time renewal chain $S_n=\sum_{i=1}^n X_i$ and its counting process
\begin{equation}\label{eq:N-def-mle}
N(m):=\sup\{n\in\mathbb{N}:S_n\le_d m\}.
\end{equation}
Under the standing assumptions $\mu_u=\mathbbm{E}[X_1^{[u]}]\in(0,\infty)$ for all $u=1,\ldots,d$,
we have $S_n^{[u]}\to\infty$ almost surely and thus $N(m)<\infty$ almost surely for every fixed $m$.

\paragraph{Age vector and censoring.}
On $\{N(m)<\infty\}$, the last renewal epoch inside the box is $S_{N(m)}$ and the age
 vector at $m$ is
\begin{equation}\label{eq:U-def-mle}
U_m:=m-S_{N(m)}\in\mathbb{N}^d.
\end{equation}
We observe the completed increments $X_1,\ldots,X_{N(m)}$ together with $U_m$, i.e.
\begin{equation}\label{eq:history-mle}
\mathcal{H}(m):=\big(X_1,\ldots,X_{N(m)},U_m\big).
\end{equation}
The next increment $X_{N(m)+1}$ is not observed. From the definition of $N(m)$ we only know that
\begin{equation}\label{eq:censoring-event}
S_{N(m)+1}\not\le_d m
\qquad\Longleftrightarrow\qquad
X_{N(m)+1}\not\le_d U_m.
\end{equation}
In particular, for $d\ge 2$ one \emph{does not} have $U_m\le_d X_{N(m)+1}$ in general; the information
in \eqref{eq:censoring-event} is only that the remaining budget is violated in \emph{at least one}
coordinate.

\noindent The multivariate censoring event has the following precise meaning.
\begin{remark}\label{rem:tail-meaning}
For $u\in\mathbb{N}^d$,
\begin{equation}\label{eq:notle-meaning}
k\not\le_d u
\qquad\Longleftrightarrow\qquad
\exists\,j\in\{1,\ldots,d\}\ \text{such that}\ k_j>u_j.
\end{equation}
Thus the censoring event $X\not\le_d u$ is a \emph{union} of coordinate exceedances, not an intersection.
For example, in $d=2$ and $u=(2,3)$,
\begin{equation*}
\{(x_1,x_2)\not\le_2(2,3)\}
=
\{x_1\ge 3\}\ \cup\ \{x_2\ge 4\},
\end{equation*}
whereas the stricter set $\{x_1\ge 3,\ x_2\ge 4\}$ corresponds to the orthant event
$\{X^{[1]}>2,\ X^{[2]}>3\}$ and is \emph{not} the event implied by \eqref{eq:censoring-event}.
\end{remark}

\paragraph{Likelihood.}
Let $\{f_\theta:\theta\in\Theta\}$ be a family of pmfs on $\mathbb{N}^d$ for the increment $X_1$.
Define the cdf with respect to the componentwise order by
\begin{equation}\label{eq:F-def-mle}
F_\theta(u):=\mathbb{P}_\theta(X_1\le_d u)=\sum_{k\le_d u} f_\theta(k),
\qquad u\in\mathbb{N}^d.
\end{equation}
The corresponding complement (upper-set probability) is
\begin{equation}\label{eq:Fbar-def-mle}
\overline{F}_\theta(u)
:=\mathbb{P}_\theta(X_1\not\le_d u)
=1-F_\theta(u)
=\sum_{k\not\le_d u} f_\theta(k),
\qquad u\in\mathbb{N}^d.
\end{equation}
Using independence of increments and \eqref{eq:censoring-event}, the likelihood based on
$\mathcal{H}(m)$ is
\begin{equation}\label{eq:likelihood-mle}
L_m(\theta)
=
\left(\prod_{n=1}^{N(m)} f_\theta(X_n)\right)\,
\overline{F}_\theta(U_m).
\end{equation}
The last factor is exactly the probability that an independent increment does \emph{not} fit inside the
remaining budget $U_m$ (i.e.\ it violates at least one coordinate constraint).

\paragraph{Counts and a complete-data approximation.}
For $x\in\mathbb{N}^d$, define the empirical counts
\begin{equation}\label{eq:counts-mle}
N_x(m):=\sum_{n=1}^{N(m)}\mathbbm{1}_{\{X_n=x\}},
\qquad
\sum_{x\in\mathbb{N}^d} N_x(m)=N(m).
\end{equation}
Writing $f$ for a generic pmf on $\mathbb{N}^d$, \eqref{eq:likelihood-mle} becomes
\begin{equation}\label{eq:likelihood-counts}
L_m(f)
=
\left(\prod_{x\in\mathbb{N}^d} f(x)^{N_x(m)}\right)\,
\overline{F}(U_m),
\qquad
\overline{F}(u)=\sum_{k\not\le_d u} f(k).
\end{equation}
Since the censoring contributes only one extra log-term, the maximizer of \eqref{eq:likelihood-counts}
is close (and asymptotically equivalent) to the maximizer of the complete-data criterion
$\sum_x N_x(m)\log f(x)$, namely
\begin{equation}\label{eq:M-est}
\widetilde{f}(x;m):=\frac{N_x(m)}{N(m)},
\qquad x\in\mathbb{N}^d,
\end{equation}
(with the convention $\widetilde{f}(\cdot;m)\equiv 0$ when $N(m)=0$).

\noindent The empirical complete-data estimator has the following finite-dimensional limit.
\begin{lemma}\label{lem:clt-ftilde}
Assume the conditions of Theorem~\ref{thm:clt-additive} and that $m\stackrel{\lambda}{\to}\infty$
as in \eqref{eq:lambda-scaling}.
Fix $r\in\mathbb{N}$ and distinct points $x^1,\ldots,x^r\in\mathbb{N}^d$.
Set 
\begin{equation*}
\widetilde{f}_r(m):=(\widetilde{f}(x^1;m),\ldots,\widetilde{f}(x^r;m))^\top \ \text{and} \ \
f_r:=(f(x^1),\ldots,f(x^r))^\top.
\end{equation*}
Then
\begin{equation}\label{eq:clt-ftilde}
\sqrt{|m|_1}\,\big(\widetilde{f}_r(m)-f_r\big)
\xrightarrow[m\stackrel{\lambda}{\to}\infty]{d}
\mathcal{N}_r(0,\Sigma_{f,r}),
\end{equation}
where, for $1\le i,j\le r$,
\begin{equation}\label{eq:Sigma-f}
\Sigma_{f,r}(i,j)
=
\frac{f(x^i)\big(\mathbbm{1}_{\{x^i=x^j\}}-f(x^j)\big)}{\mu_\lambda}.
\end{equation}
\end{lemma}

\begin{proof}
For $i=1,\ldots,r$ define $h_i(x):=\mathbbm{1}_{\{x=x^i\}}-f(x^i)$ and set
$h:=(h_1,\ldots,h_r)^\top$.
Then $\mathbbm{E}[h(X_1)]=0$ and $\mathbbm{E}[\|h(X_1)\|^2]<\infty$.
Moreover,
\begin{equation*}
\sum_{n=1}^{N(m)}h(X_n)
=
\big(N_{x^1}(m)-f(x^1)N(m),\ldots,N_{x^r}(m)-f(x^r)N(m)\big)^\top.
\end{equation*}
By Theorem~\ref{thm:clt-additive},
\begin{equation*}
\frac{1}{\sqrt{|m|_1}}\sum_{n=1}^{N(m)}h(X_n)
\xrightarrow[m\stackrel{\lambda}{\to}\infty]{d}
\mathcal{N}_r\!\big(0,\mu_\lambda\,\Sigma_h\big),
\end{equation*}
where $\Sigma_h=\mathrm{Cov}(h(X_1))$.
Since
\begin{equation*}
\widetilde{f}(x^i;m)-f(x^i)=\big(N_{x^i}(m)-f(x^i)N(m)\big)/N(m), 
\end{equation*}
we have
\begin{equation*}
\sqrt{|m|_1}\,\big(\widetilde{f}_r(m)-f_r\big)
=
\frac{|m|_1}{N(m)}\cdot
\frac{1}{\sqrt{|m|_1}}\sum_{n=1}^{N(m)}h(X_n).
\end{equation*}
Using $N(m)/|m|_1\to\mu_\lambda$ almost surely (Proposition~\ref{prop:SLLN-N}), we get
$|m|_1/N(m)\to 1/\mu_\lambda$, and the limit covariance is $\Sigma_h/\mu_\lambda$.
A direct computation yields \eqref{eq:Sigma-f}.
\end{proof}

\paragraph{Exact nonparametric MLE under fixed-horizon censoring.}
Let $u:=U_m$ and define the censoring set (an upper set with respect to $\le_d$)
\begin{equation}\label{eq:A-u-def}
A(u):=\{k\in\mathbb{N}^d:k\not\le_d u\},
\qquad
N_{A(u)}(m):=\sum_{k\in A(u)}N_k(m).
\end{equation}
Maximizing \eqref{eq:likelihood-counts} is equivalent to maximizing the concave criterion
\begin{equation}\label{eq:loglik-np}
\ell_m(f)
=
\sum_{k\in\mathbb{N}^d}N_k(m)\log f(k)
+
\log\!\left(\sum_{k\in A(u)}f(k)\right),
\end{equation}
over pmfs $f$ on $\mathbb{N}^d$.

\noindent The exact nonparametric maximum likelihood estimator is given below.
\begin{proposition}\label{prop:exact-mle}
Let $n:=N(m)$ and $u:=U_m$.
If $N_{A(u)}(m)>0$, the criterion \eqref{eq:loglik-np} is uniquely maximized by the pmf
$\widehat{f}(\cdot\,;m)$ given by
\begin{equation}\label{eq:fhat-exact}
\widehat{f}(k;m)
=
\begin{cases}
\dfrac{N_k(m)}{n+1},
& k\le_d u,\\[1.0em]
\dfrac{N_k(m)}{n+1}\left(1+\dfrac{1}{N_{A(u)}(m)}\right),
& k\not\le_d u.
\end{cases}
\end{equation}
Equivalently, for $k\not\le_d u$,
\begin{equation*}
\widehat{f}(k;m)
=
\frac{N_k(m)}{n+1}
+\frac{1}{n+1}\cdot\frac{N_k(m)}{N_{A(u)}(m)}.
\end{equation*}
If $N_{A(u)}(m)=0$, any pmf satisfying $\widehat{f}(k;m)=N_k(m)/(n+1)$ for $k\le_d u$ and
$\sum_{k\not\le_d u}\widehat{f}(k;m)=1/(n+1)$ is a maximizer.
\end{proposition}

\begin{proof}
Write $T(f):=\sum_{k\in A(u)}f(k)$ and introduce the Lagrangian
\begin{equation*}
\mathcal{L}(f,\eta)
=
\sum_{k\in\mathbb{N}^d}N_k(m)\log f(k)
+
\log T(f)
+
\eta\left(1-\sum_{k\in\mathbb{N}^d}f(k)\right),
\end{equation*}
where $\eta$ is a Lagrange multiplier.
For $k\le_d u$ the first-order condition yields $N_k(m)/f(k)-\eta=0$, hence $f(k)=N_k(m)/\eta$.
For $k\not\le_d u$ it yields $N_k(m)/f(k)+1/T(f)-\eta=0$.

Assume first that $N_{A(u)}(m)>0$ and set $n_A:=N_{A(u)}(m)$.
Solving the first-order conditions shows that
$f(k)=N_k(m)/\eta$ for $k\le_d u$ and
$f(k)=N_k(m)/(\eta-1/T(f))$ for $k\not\le_d u$.
Summing over $k\not\le_d u$ yields
$T(f)=n_A/(\eta-1/T(f))$, i.e.\ $\eta\,T(f)-1=n_A$ and therefore $T(f)=(n_A+1)/\eta$.
Using $\sum_k f(k)=1$ gives
$1=(n-n_A)/\eta+T(f)=(n+1)/\eta$, so $\eta=n+1$ and $T(f)=(n_A+1)/(n+1)$.
Substituting back yields \eqref{eq:fhat-exact}.
Uniqueness follows from strict concavity on the simplex when $n_A>0$.

If $n_A=0$, the criterion depends on $(f(k))_{k\not\le_d u}$ only through their sum $T(f)$.
The constraint $\sum_k f(k)=1$ forces $T(f)=1/(n+1)$, yielding the stated non-uniqueness.
\end{proof}

\noindent The exact estimator and the complete-data estimator differ only by a boundary correction.
\begin{remark}\label{rem:fhat-ftilde}
Let $n=N(m)$ and $u=U_m$.
When $N_{A(u)}(m)>0$,
\begin{eqnarray*}
\widehat{f}(k;m)
&=&
\widetilde{f}(k;m)\,\frac{n}{n+1}
\quad\text{for }k\le_d u, \\
\widehat{f}(k;m)
&=&
\widetilde{f}(k;m)\,\frac{n}{n+1}\left(1+\frac{1}{N_{A(u)}(m)}\right)
\quad\text{for }k\not\le_d u.
\end{eqnarray*}
In particular, for each fixed $k$, $|\widehat{f}(k;m)-\widetilde{f}(k;m)|=\mathcal{O}(1/n)$.
\end{remark}

\noindent The exact estimator has the same first-order limit as the complete-data estimator.
\begin{theorem}\label{thm:clt-fhat}
Under the assumptions of Lemma~\ref{lem:clt-ftilde}, the exact MLE $\widehat{f}$ satisfies the same
finite-dimensional CLT: for any distinct $x^1,\ldots,x^r\in\mathbb{N}^d$,
\begin{equation*}
\sqrt{|m|_1}\,\big(\widehat{f}_r(m)-f_r\big)
\xrightarrow[m\stackrel{\lambda}{\to}\infty]{d}
\mathcal{N}_r(0,\Sigma_{f,r}),
\end{equation*}
where $\widehat{f}_r(m):=(\widehat{f}(x^1;m),\ldots,\widehat{f}(x^r;m))^\top$ and
$\Sigma_{f,r}$ is given by \eqref{eq:Sigma-f}.
\end{theorem}

\begin{proof}
By Remark~\ref{rem:fhat-ftilde} and $N(m)/|m|_1\to\mu_\lambda>0$, we have $N(m)\to\infty$ and thus
$\sqrt{|m|_1}\,\|\widehat{f}_r(m)-\widetilde{f}_r(m)\|\to 0$ in probability.
The conclusion follows from Lemma~\ref{lem:clt-ftilde} and Slutsky's lemma.
\end{proof}

\noindent The same estimation scheme can be applied to the embedded batch-free chain.
\begin{remark}
If the model allows instantaneous renewals (i.e.\ $f(0_d)>0$), one may alternatively work with the
embedded (batch-free) increment distribution $f^{*}$ defined in Remark~\ref{rem:no-batch}.
The estimators above can be applied to $f^{*}$ after removing the zero increments from the observed
sequence, and then converted back to an estimator for $f$ via
$f=p_0\,\ez+(1-p_0)f^{*}$ with $p_0=f(0_d)$.
\end{remark}

\section{Applications}\label{sec:Applications}

This section illustrates several applications of the multi-index convolution and multi-time renewal framework.
We first derive a binomial--multiset identity as a direct specialization of the inverse-coefficient
formula in Theorem~\ref{thm:inverse-exists}, and we explain how it can be used as a simple regression
test for convolutional inversion routines.
We then show how multi-time renewal functions provide closed expressions for expected costs under
two-attribute (calendar--usage) \emph{non-renewing} warranty policies.
Finally, we present (i)~a bivariate alternating-renewal availability computation in discrete time
and (ii)~a discretization-based method for approximating continuous-time bivariate renewal and
availability quantities.

\subsection{A binomial--multiset identity}\label{subsec:comb-identity}

The coefficientwise inverse formula in Theorem~\ref{thm:inverse-exists} yields a concrete
binomial--multiset identity.

Let $\sd:\mathbb{N}^d\to\mathbb{R}$ be the constant-one sequence, $\sd(k)\equiv 1$, and let
$\delta_d:=\sd^{(-1)}$ be its convolutional inverse. By Example~\ref{ex:summation-difference},
\begin{equation}\label{eq:delta-d-explicit}
\delta_d(k)=
\begin{cases}
(-1)^{|k|_1}, & k\in\{0,1\}^d,\\
0, & \text{otherwise}.
\end{cases}
\end{equation}

On the other hand, applying \eqref{eq:inverse-binomial} in Theorem~\ref{thm:inverse-exists} to
$a=\sd$ (note that $\sd(0_d)=1$ and $\widetilde a=\sd$) gives, for every $k\in\mathbb{N}^d$,
\begin{equation}\label{eq:delta-d-inverse-sum}
\delta_d(k)
=
\sum_{n=0}^{|k|_1}
(-1)^n
\binom{|k|_1+1}{n+1}\,
\sd^{(n)}(k).
\end{equation}
By Proposition~\ref{prop:sd-nfold}, for $n\ge 1$,
\begin{equation}\label{eq:sd-nfold-product}
\sd^{(n)}(k)=\prod_{u=1}^d \multiset{n}{k_u}.
\end{equation}
Also, $\sd^{(0)}=\ez$, so $\sd^{(0)}(k)=\mathbbm{1}_{\{k=0_d\}}$.
Equivalently, \eqref{eq:sd-nfold-product} continues to hold for $n=0$ if one adopts the convention
$\multiset{0}{0}=1$ and $\multiset{0}{m}=0$ for $m\ge 1$.

Substituting \eqref{eq:sd-nfold-product} into \eqref{eq:delta-d-inverse-sum} and using
\eqref{eq:delta-d-explicit} yields the $d$-dimensional identity
\begin{equation}\label{eq:multiindex-comb-identity}
\sum_{n=0}^{|k|_1}
(-1)^n
\binom{|k|_1+1}{n+1}\,
\prod_{u=1}^d \multiset{n}{k_u}
=
\begin{cases}
(-1)^{|k|_1}, & k\in\{0,1\}^d,\\
0, & \text{otherwise}.
\end{cases}
\end{equation}
For $d=1$, \eqref{eq:multiindex-comb-identity} reduces to the classical univariate binomial--multiset identity.

\paragraph{Computational regression test (explicit link to inversion routines).}
Identity~\eqref{eq:multiindex-comb-identity} is not only a combinatorial by-product: it provides a
useful sanity check for implementations of convolutional inversion.
Indeed, for $a=\sd$ the inverse is known explicitly from Example~\ref{ex:summation-difference}, namely
$\sd^{(-1)}=\delta_d$ with support contained in $\{0,1\}^d$; at the same time,
Theorem~\ref{thm:inverse-exists} and Proposition~\ref{prop:sd-nfold} imply the representation
\eqref{eq:multiindex-comb-identity}. Thus any routine that computes $a^{(-1)}$
(e.g.\ the recursion of Proposition~\ref{prop:inv-recursion} or the Newton update \eqref{eq:newton-multi})
should reproduce $\delta_d$ on the computed grid.

Concretely, fix a truncation box $\{0,\ldots,K\}^d$ and:
(i) compute an approximation $b$ to $\sd^{(-1)}$ on this box;
(ii) verify coefficientwise that $b(k)=\delta_d(k)$ for all $k$ in the box; and
(iii) multiply back and check the residual
\begin{equation}\label{eq:inverse-check-residual}
(\sd*b)(k)-\ez(k)=0,
\qquad
k\in\{0,\ldots,K\}^d.
\end{equation}
If FFTs are used to evaluate \eqref{eq:inverse-check-residual}, the convolution must be computed as a
\emph{linear} convolution (hence requires zero-padding to prevent circular aliasing; see
Remark~\ref{rem:fft-cost} and Appendix~\ref{app:fft}).

\subsection{Application to a two-dimensional non-renewing warranty model}\label{subsec:warranty}

We illustrate how the multi-time renewal function can be used to evaluate expected costs under
two-attribute (calendar--usage) \emph{non-renewing} warranty policies, following the bivariate
point-process viewpoint of \cite{murthy1995two} (see also \cite{jack2009repair} for
related usage-rate strategies).

\paragraph{Renewal model and claim count.}
Let $X=(X^{[1]},X^{[2]})$ denote the bivariate lifetime of a product, where $X^{[1]}$ is the
calendar age (in months) and $X^{[2]}$ is the accumulated usage measured in discrete blocks
(e.g.\ one block corresponds to $50$ hours of use). After each failure, the item is replaced
(or perfectly repaired) so that successive lifetimes are i.i.d.\ with the same law as $X$.

Let $(X_n)_{n\ge 1}$ be i.i.d.\ copies of $X$ and define the renewal epochs
\begin{equation*}
S_0=0_2,
\qquad
S_n=\sum_{i=1}^n X_i,
\quad n\ge 1.
\end{equation*}
For a warranty coverage region $\mathcal{R}\subset\mathbb{N}^2$, define the number of covered claims by
\begin{equation*}
N_{\mathcal{R}}
:=
\sum_{n\ge 1}\mathbbm{1}_{\{S_n\in\mathcal{R}\}},
\qquad
M_{\mathcal{R}}:=\mathbbm{E}[N_{\mathcal{R}}].
\end{equation*}
Assuming a constant claim cost $c>0$, the expected warranty cost is $C_{\mathcal{R}}:=c\,M_{\mathcal{R}}$.

\paragraph{Sojourn-time distribution.}
In the numerical illustration below, we assume that $X$ follows the bivariate discrete Weibull
model of Appendix~\ref{subsec:DISC} (see Definition~\ref{def:bdw}).
Its pmf $f$ is obtained from its survival representation via the mixed finite-difference identity
\eqref{eq:pmf-from-survival}.
We take $q_1=0.2$, $q_2=0.3$, $\theta=0.7$, $\beta_1=1$, $\beta_2=2$ and $c=5$.

\paragraph{Policy A (rectangular warranty).}
For limits $k_1,k_2\in\mathbb{N}$, consider the rectangle
$\mathcal{R}_A(k_1,k_2)=\{0,\ldots,k_1\}\times\{0,\ldots,k_2\}$.
Because $(S_n)_{n\ge 0}$ is coordinatewise nondecreasing, one has
$N_{\mathcal{R}_A}(k_1,k_2)=N(k_1,k_2)$ and $M_{\mathcal{R}_A}(k_1,k_2)=M(k_1,k_2)$,
where $M$ is the bivariate renewal function of Section~\ref{sec:multitime-renewal-theory}.

We compute $M$ from the explicit representation
$M=(\ez-f)^{(-1)}*F$ (Corollary~\ref{cor:M-explicit})
using FFT-accelerated convolutions (Appendix~\ref{app:fft}) and convolutional inversion
(Proposition~\ref{prop:inv-recursion} or Newton iteration \eqref{eq:newton-multi}).
Table~\ref{tab:warrantyA} reports $C_A(k_1,k_2)=c\,M(k_1,k_2)$.

\begin{table}[h!]
\centering
\begin{tabular}{r|rrrrr}
\hline
       & \multicolumn{5}{c}{$k_2$} \\
$k_1$  & 5        & 10       & 20        & 50        & 100       \\
\hline
5   & 17.01078 & 19.99986 & 20.00000 & 20.00000 & 20.00000 \\
10  & 18.84488 & 35.50032 & 39.99999 & 40.00000 & 40.00000 \\
20  & 18.85741 & 37.96875 & 73.04510 & 80.00000 & 80.00000 \\
50  & 18.85741 & 37.96918 & 76.19198 & 186.94410 & 200.00000 \\
100 & 18.85741 & 37.96918 & 76.19198 & 190.86040 & 377.87600 \\
\hline
\end{tabular}
\caption{Expected warranty cost under Policy A: $C_A(k_1,k_2)=c\,M(k_1,k_2)$.}\label{tab:warrantyA}
\end{table}

\paragraph{Policy B (coverage until both limits are exceeded).}
Here coverage is guaranteed for at least $k_1$ months and at least $k_2$ usage blocks,
and ends only when \emph{both} coordinates exceed their limits. Equivalently, the coverage region is
\begin{equation*}
\mathcal{R}_B(k_1,k_2)
=
\big(\{0,\ldots,k_1\}\times\mathbb{N}\big)\,\cup\,
\big(\mathbb{N}\times\{0,\ldots,k_2\}\big).
\end{equation*}
Let $N^{[1]}(k_1)$ and $N^{[2]}(k_2)$ be the marginal counting processes, and recall that
$N(k_1,k_2)=\min\{N^{[1]}(k_1),N^{[2]}(k_2)\}$.
Since a claim at epoch $S_n$ is covered whenever $S_n^{[1]}\le k_1$ or $S_n^{[2]}\le k_2$,
the number of covered failures satisfies
\begin{equation*}
N_B(k_1,k_2)
=
\max\{N^{[1]}(k_1),\,N^{[2]}(k_2)\}
=
N^{[1]}(k_1)+N^{[2]}(k_2)-N(k_1,k_2).
\end{equation*}
Taking expectations gives
\begin{equation*}
M_B(k_1,k_2)
=
M^{[1]}(k_1)+M^{[2]}(k_2)-M(k_1,k_2),
\qquad
C_B(k_1,k_2)=c\,M_B(k_1,k_2).
\end{equation*}
Table~\ref{tab:warrantyB} reports $C_B$.

\begin{table}[h!]
\centering
\begin{tabular}{r|rrrrr}
\hline
       & \multicolumn{5}{c}{$k_2$} \\
$k_1$  & 5        & 10       & 20        & 50        & 100       \\
\hline
5   & 21.84664 & 37.97033 & 76.19198 & 190.86040 & 381.97440 \\
10  & 40.01254 & 42.46886 & 76.19199 & 190.86040 & 381.97440 \\
20  & 80.00000 & 80.00043 & 83.14688 & 190.86040 & 381.97440 \\
50  & 200.00000 & 200.00000 & 200.00000 & 203.91630 & 381.97440 \\
100 & 400.00000 & 400.00000 & 400.00000 & 400.00000 & 404.09840 \\
\hline
\end{tabular}
\caption{Expected warranty cost under Policy B.}\label{tab:warrantyB}
\end{table}

\paragraph{Policy C (two-rectangle extension).}
Given $k_1<k_2$ and $\ell_1<\ell_2$, consider the L-shaped region
\begin{equation*}
\mathcal{R}_C
=
\mathcal{R}_A(k_1,\ell_2)\cup \mathcal{R}_A(k_2,\ell_1).
\end{equation*}
By inclusion--exclusion and the identity
$N\big(\min\{(k_1,\ell_2),(k_2,\ell_1)\}\big)=N(k_1,\ell_1)$, the corresponding claim count satisfies
\begin{equation*}
N_C
=
N(k_1,\ell_2)+N(k_2,\ell_1)-N(k_1,\ell_1)
=
\max\{N(k_1,\ell_2),\,N(k_2,\ell_1)\}.
\end{equation*}
Therefore,
\begin{equation*}
M_C
=
M(k_1,\ell_2)+M(k_2,\ell_1)-M(k_1,\ell_1),
\qquad
C_C=c\,M_C.
\end{equation*}
Table~\ref{tab:warrantyC} reports $C_C$ for the choice
$k_1=\lfloor 0.5\,k_2\rfloor$ and $\ell_1=\lfloor 0.5\,\ell_2\rfloor$.

\begin{table}[h!]
\centering
\scriptsize
\begin{tabular}{r|rrrrr}
\hline
       & \multicolumn{5}{c}{$\ell_2$} \\
$k_2$  & 5        & 10       & 20        & 50        & 100       \\
\hline
5   & 4.782576 & 13.37002 & 16.00000 & 16.00000 & 16.00000 \\
10  & 13.37744 & 17.65929 & 34.88300 & 36.00000 & 36.00000 \\
20  & 15.03626 & 33.55673 & 46.56045 & 88.15900 & 96.00000 \\
50  & 15.03685 & 34.14690 & 88.15903 & 99.32210 & 183.13420 \\
100 & 15.03685 & 34.14690 & 91.48110 & 183.13420 & 199.90400 \\
\hline
\end{tabular}
\caption{Expected warranty cost under Policy C with $k_1=\lfloor 0.5\,k_2\rfloor$ and $\ell_1=\lfloor 0.5\,\ell_2\rfloor$.}
\label{tab:warrantyC}
\end{table}

\paragraph{Computational validation.}
For each policy we validate the computed inversion/convolution routines by the
multiply-back residual check $(\ez-f)*(\ez-f)^{(-1)}=\ez$ on the working grid, and (for renewal
functions) by verifying monotonicity of $M(\cdot)$ in each coordinate
(up to small floating-point tolerances).

\subsection{Reliability and availability modelling in alternating renewal systems}\label{subsec:alt-renewal}

A repairable system often alternates between two mutually exclusive states:
an \emph{up} state (operational) and a \emph{down} state (under repair/maintenance).
In many applications, deterioration and maintenance depend simultaneously on more than one
exposure variable, such as calendar age and accumulated usage.
We model this on the bivariate grid $k_{1:2}=(k_1,k_2)\in\mathbb{N}^2$.

Let $(U_n)_{n\ge 1}$ and $(D_n)_{n\ge 1}$ be independent i.i.d.\ sequences of $\mathbb{N}^2$-valued
random vectors, where
$U_n=(U_n^{[1]},U_n^{[2]})$ and $D_n=(D_n^{[1]},D_n^{[2]})$ denote the $n$-th up-time and down-time
vector, respectively.
Define the cycle duration
\begin{equation*}
C_n:=U_n+D_n,
\qquad
S_0:=0_2,
\qquad
S_n:=\sum_{j=1}^n C_j,
\quad n\ge 1,
\end{equation*}
so that $(S_n)$ are renewal epochs of completed up--down cycles.

We use componentwise order: for $x,y\in\mathbb{R}^2$, write $x\le_2 y$ if $x_u\le y_u$ for $u=1,2$,
and write $x<_2 y$ if $x_u<y_u$ for $u=1,2$.
At index $k\in\mathbb{N}^2$, the system is operational if and only if there exists $n\ge 0$ such that
\begin{equation*}
S_n \le_2 k <_2 S_n + U_{n+1}.
\end{equation*}
The (point) availability function is therefore
\begin{equation}\label{eq:availability-def}
A(k)
:=
\mathbb{P}\big[\text{system is operational at }k\big]
=
\sum_{n=0}^{\infty}\mathbb{P}\big[S_n\le_2 k <_2 S_n+U_{n+1}\big],
\
k\in\mathbb{N}^2.
\end{equation}
If $C_1$ satisfies $\mathbb{P}(C_1=0_2)=0$, then $|S_n|_1\ge n$ and only finitely many terms in
\eqref{eq:availability-def} can be nonzero for any fixed $k$; otherwise the series is infinite but
well-defined by monotone convergence.

\paragraph{Renewal equation.}
Let $f_C$ denote the pmf of $C_1$ (so $f_C=f_U*f_D$ in the multi-index convolution sense).
Define the \emph{bivariate survival-on-the-grid} sequence of the up-time by
\begin{equation}\label{eq:SU-def}
S_U(k)
:=
\mathbb{P}\big(U_1>_2 k\big)
=
\mathbb{P}\big(U_1^{[1]}>k_1,\ U_1^{[2]}>k_2\big),
\qquad
k\in\mathbb{N}^2.
\end{equation}
Note that $S_U$ is the standard bivariate survival function; in general $S_U(k)\neq 1-F_U(k)$
when $U_1^{[1]}$ and $U_1^{[2]}$ are dependent.

Conditioning on the first cycle length $C_1$ yields the bivariate renewal equation
\begin{equation}\label{eq:availability-renewal-eq}
A = S_U + f_C*A,
\end{equation}
and hence, by the general solution of multi-time renewal equations
(Theorem~\ref{thm:general-renewal-solution}),
\begin{equation}\label{eq:availability-solution}
A=(\ez-f_C)^{(-1)}*S_U.
\end{equation}
This expresses availability as a convolution between a renewal structure (through $(\ez-f_C)^{(-1)}$)
and an ``input'' sequence $S_U$ describing up-time reliability.

\paragraph{Numerical illustration (bivariate discrete Weibull cycles).}
We illustrate \eqref{eq:availability-solution} when $U_n$ and $D_n$ follow the bivariate discrete Weibull
model (Definition~\ref{def:bdw}). Specifically, we take
\begin{equation*}
q^{U}=(0.8,0.75),
\
\beta^{U}=(1,1),
\
\theta^{U}=0.95,
\
q^{D}=(0.7,0.6),
\
\beta^{D}=(1,1),
\
\theta^{D}=0.8.
\end{equation*}
Table~\ref{tab:avail-discrete-weibull} reports values of $A(k_1,k_2)$ for representative
calendar/usage combinations.
\begin{table}[h!]
\centering
\begin{tabular}{|c|ccccc|}
\hline
$A(k_{1:2})$ & 2 & 20 & 80 & 520 & 1040 \\
\hline
2    & 0.7000 & 0.001276 & $1.88\times10^{-12}$ & $8.31\times10^{-18}$ & $1.79\times10^{-17}$ \\
20   & $4.08\times10^{-3}$ & 0.06938 & $2.79\times10^{-5}$ & $2.16\times10^{-17}$ & $3.75\times10^{-17}$ \\
80   & $2.88\times10^{-10}$ & $1.36\times10^{-3}$ & 0.02313 & $3.80\times10^{-17}$ & $3.62\times10^{-17}$ \\
520  & $1.30\times10^{-18}$ & $9.41\times10^{-18}$ & $4.64\times10^{-17}$ & $1.58\times10^{-3}$ & $1.27\times10^{-16}$ \\
1040 & $1.12\times10^{-17}$ & $1.82\times10^{-17}$ & $3.54\times10^{-17}$ & $8.37\times10^{-7}$ & $1.42\times10^{-4}$ \\
\hline
\end{tabular}
\caption{Availability values for the discrete bivariate Weibull alternating-renewal model.}
\label{tab:avail-discrete-weibull}
\end{table}
\subsection{Approximation of the continuous-time bivariate renewal function}\label{subsec:cont-approx}

\paragraph{Continuous-time bivariate renewal equation.}
Let $(X_n)_{n\ge 1}$ be an i.i.d.\ sequence of $\mathbb{R}_+^2$-valued random vectors,
$X_n=(X_n^{[1]},X_n^{[2]})$, with joint cdf $F$ and (when it exists) joint density $f$.
Define the renewal epochs
\begin{equation*}
S_0:=0_2,
\qquad
S_n:=\sum_{j=1}^n X_j=(S_n^{[1]},S_n^{[2]}),
\quad n\ge 1,
\end{equation*}
and the bivariate counting process
\begin{equation*}
N(x_1,x_2):=\sup\{n\in\mathbb{N}:\ S_n\le_2 (x_1,x_2)\}.
\end{equation*}
The bivariate renewal function is $M(x_1,x_2):=\mathbb{E}[N(x_1,x_2)]$.
Hunter's bivariate renewal theory shows that $M$ satisfies the renewal integral equation
\begin{equation}\label{eq:renewal-integral}
M(x_{1:2})
=
F(x_{1:2})
+
\int_{0}^{x_1}\int_{0}^{x_2}
M(x_1-t_1,x_2-t_2)\, dF(t_1,t_2),
\qquad x_{1:2}\in\mathbb{R}_+^2.
\end{equation}
Introduce the bivariate Riemann--Stieltjes convolution operator
\begin{equation*}
[G\star F](x_{1:2})
:=
\int_{0}^{x_1}\int_{0}^{x_2}
G(x_1-t_1,x_2-t_2)\, dF(t_1,t_2),
\end{equation*}
so that \eqref{eq:renewal-integral} reads $M=F+M\star F$.
Iterating yields the series representation
\begin{equation}\label{eq:renewal-series}
M(x_{1:2})=\sum_{n=1}^{\infty}F^{\star n}(x_{1:2}),
\end{equation}
where $F^{\star n}$ denotes the $n$-fold Stieltjes convolution (the cdf of $S_n$).

\paragraph{Grid discretization and reduction to the discrete convolution framework.}
Fix step sizes $h_1,h_2>0$ and consider the grid points $x_{1:2}=(h_1k_1,h_2k_2)$ with $k_{1:2}\in\mathbb{N}^2$.
Define the discretized pmf on $\mathbb{N}^2$ by the rectangle increment of $F$:
\begin{equation}\label{eq:fh-cell}
\begin{aligned}
f_{h_{1:2}}(k_{1:2})
:={}&
F(h_1k_1,h_2k_2)
-F(h_1(k_1-1),h_2k_2) \\
&-F(h_1k_1,h_2(k_2-1))
+F(h_1(k_1-1),h_2(k_2-1)),
\ \ k_{1:2}\in\mathbb{N}^2,
\end{aligned}
\end{equation}
with the convention $F(x_1,x_2)=0$ if $x_1<0$ or $x_2<0$.
This is the distribution of the discretized vector
$X^{h_{1:2}}=(\lceil X^{[1]}/h_1\rceil,\lceil X^{[2]}/h_2\rceil)$.
Also define the sampled cdf sequence
\begin{equation*}
F_{h_{1:2}}(k_{1:2})
:=
F(h_1k_1,h_2k_2)
=
[\sd*f_{h_{1:2}}](k_{1:2}).
\end{equation*}

Replacing the bivariate Stieltjes convolution in \eqref{eq:renewal-integral} by its grid version
based on the cell probabilities \eqref{eq:fh-cell} leads, at grid points, to the discrete renewal equation
\begin{equation}\label{eq:renewal-discrete-approx}
M_{h_{1:2}} = F_{h_{1:2}} + f_{h_{1:2}} * M_{h_{1:2}}.
\end{equation}
Its unique solution is
\begin{equation}\label{eq:Mh-solution}
M_{h_{1:2}}
=
(\ez-f_{h_{1:2}})^{(-1)}*F_{h_{1:2}}.
\end{equation}
Thus, computing $M(\cdot)$ on a grid is reduced to computing discrete multi-index convolutions and
convolutional inverses, as developed in Section~\ref{sec:convolution}.

\paragraph{Numerical illustration (FGM bivariate exponential).}
In the numerical results that follow, we evaluate the diagonal ratio $M(t,t)/t$ up to $t=16$
using the discretization \eqref{eq:Mh-solution} with $h_1=h_2=0.01$ for the FGM bivariate exponential model
(Definition~\ref{def:fgm-exp}). We consider $\rho\in\{-0.8,0,0.8\}$ and three parameter settings:
(i) $(\lambda_1,\lambda_2)=(1,1)$,
(ii) $(\lambda_1,\lambda_2)=(5,4)$,
(iii) $(\lambda_1,\lambda_2)=(6,7)$.

Along the diagonal, the strong law for the marginal renewal processes implies
\begin{equation*}
\lim_{t\to\infty}\frac{M(t,t)}{t}
=
\min\left\{\frac{1}{\mu_1},\frac{1}{\mu_2}\right\}
=
\frac{1}{\max\{\mu_1,\mu_2\}},
\qquad
\mu_u:=\mathbb{E}[X_1^{[u]}].
\end{equation*}
For each fixed $(\lambda_1,\lambda_2)$, the curves corresponding to $\rho\in\{-0.8,0,0.8\}$
converge to the same asymptotic limit; the effect of dependence is mainly visible in the
transient regime. The figures plot $\log(M(t,t)/t)$ in order to emphasize these differences.

\begin{figure}[H]
  \centering
  \includegraphics[width=1\textwidth]{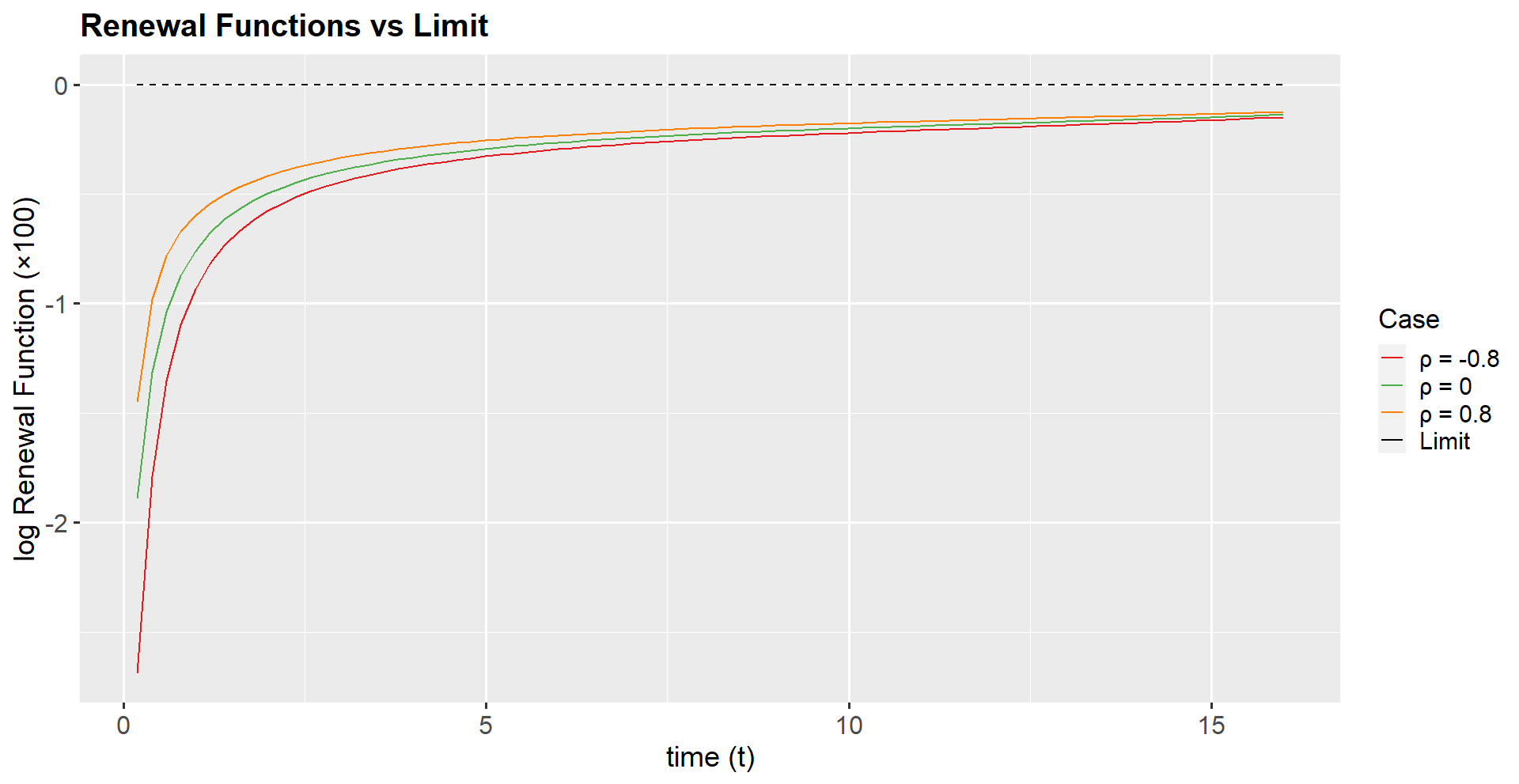}
  \caption{Logarithm of the diagonal ratio $M(t,t)/t$ for the FGM bivariate exponential model: $(\lambda_1,\lambda_2)=(1,1)$.}
  \label{fig:renewal-limit1}
\end{figure}

\begin{figure}[H]
  \centering
  \includegraphics[width=1\textwidth]{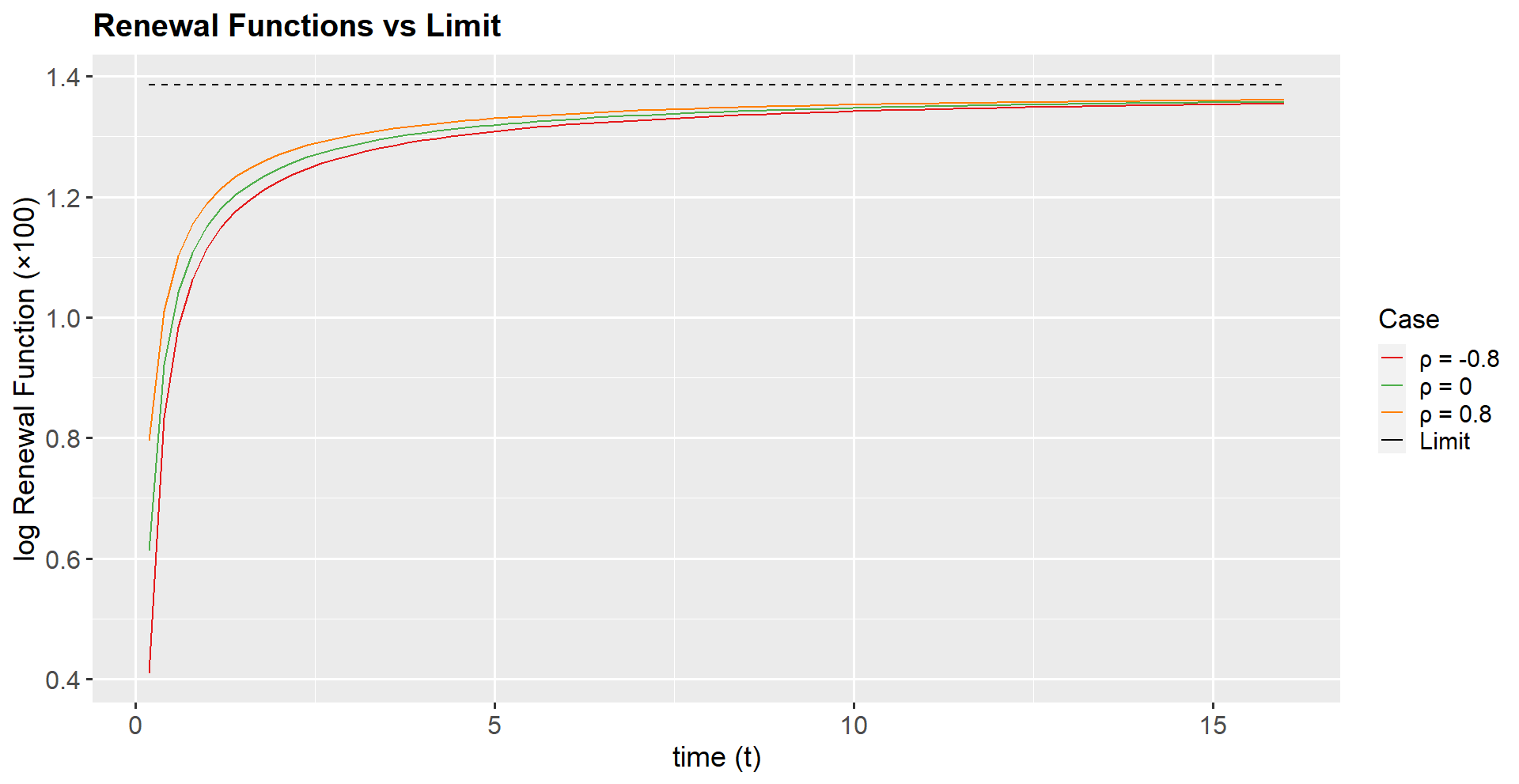}
  \caption{Logarithm of the diagonal ratio $M(t,t)/t$ for the FGM bivariate exponential model: $(\lambda_1,\lambda_2)=(5,4)$.}
  \label{fig:renewal-limit2}
\end{figure}

\begin{figure}[H]
  \centering
  \includegraphics[width=1\textwidth]{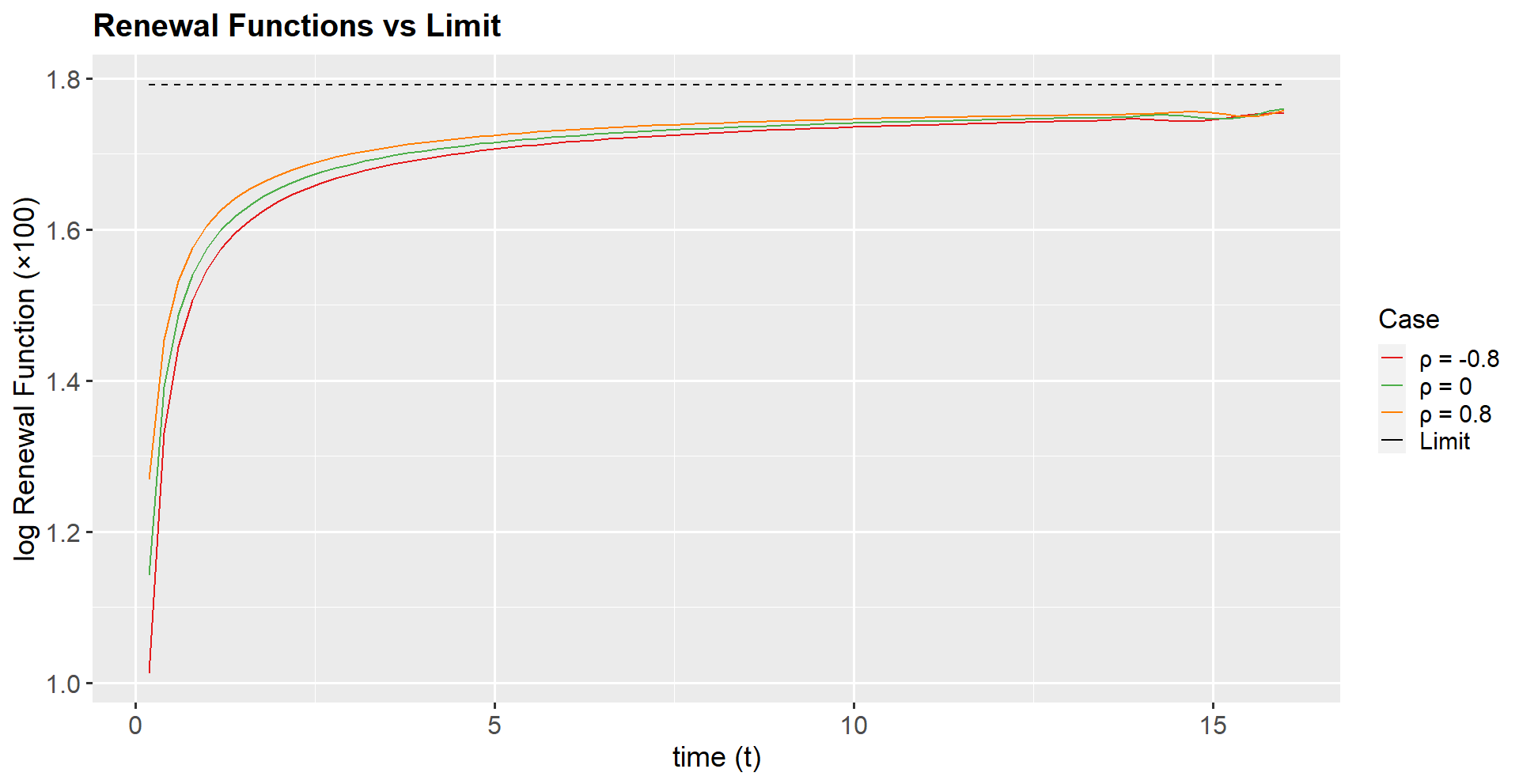}
  \caption{Logarithm of the diagonal ratio $M(t,t)/t$ for the FGM bivariate exponential model: $(\lambda_1,\lambda_2)=(6,7)$.}
  \label{fig:renewal-limit3}
\end{figure}

\paragraph{Continuous-time bivariate availability and its discrete approximation.}
A second quantity of interest is the availability of a continuous-time alternating renewal system on $\mathbb{R}_+^2$.
Let $(U_n)_{n\ge 1}$ and $(D_n)_{n\ge 1}$ be independent i.i.d.\ sequences in $\mathbb{R}_+^2$ with cdfs $F_U,F_D$,
and define $C_n:=U_n+D_n$ and $S_n:=\sum_{j=1}^n C_j$.
The point availability at $x_{1:2}\in\mathbb{R}_+^2$ is
\begin{equation*}
A(x_{1:2})
\!:=\!
\mathbb{P}\big[\exists n\ge 0: S_n\le_2 x_{1:2} <_2 S_n+U_{n+1}\big]
\!=\!
\!\sum_{n=0}^{\infty}\!\mathbb{P}\big[S_n\le_2 x_{1:2} <_2 S_n+U_{n+1}\!\big].
\end{equation*}
It satisfies the bivariate renewal integral equation (see \cite{yang2001bivariate})
\begin{equation}\label{eq:avail-cont}
A(x_{1:2})
=
S_U(x_{1:2})
+
[A\star F_C](x_{1:2}),
\end{equation}
where $S_U(x_{1:2})=\mathbb{P}(U_1>_2 x_{1:2})$ is the bivariate survival function of $U_1$ and $F_C$ is the cdf of $C_1$.

Following the discretization principle above (see also \cite{sarada2021bi}),
sample on the grid $x_{1:2}=(h_1k_1,h_2k_2)$:
\begin{equation*}
S_{U,h_{1:2}}(k_{1:2})
:=
S_U(h_1k_1,h_2k_2),
\end{equation*}
and define $f_{h_{1:2},C}$ from $F_C$ via the rectangle increment formula \eqref{eq:fh-cell}.
Then the discrete approximation of availability is obtained from the discrete renewal equation
\begin{equation*}
A_{h_{1:2}}
=
S_{U,h_{1:2}} + f_{h_{1:2},C}*A_{h_{1:2}},
\qquad\text{so that}\qquad
A_{h_{1:2}}=(\ez-f_{h_{1:2},C})^{(-1)}*S_{U,h_{1:2}}.
\end{equation*}

\paragraph{Monte Carlo check.}
To validate the discretization numerically at a point $x_{1:2}$, one may simulate i.i.d.\ alternating cycles
$(U_n,D_n)$, form $C_n$ and $S_n$, stop at the largest $n$ such that $S_n\le_2 x_{1:2}$, and record whether
$x_{1:2}<_2 S_n+U_{n+1}$. Averaging this indicator over many independent replications yields a Monte Carlo estimate of $A(x_{1:2})$.

We apply the methodology to the Moran--Downton bivariate exponential distribution (Definition~\ref{def:downton-exp}).
Availability values are computed using both the discrete renewal scheme and a Monte Carlo simulation with $10{,}000$
independent replications, evaluated along the diagonal $(x,x)$, $x=1,2,4,8,16$, with step size $h=0.01$:

\begin{table}[h!]
\centering
\begin{tabular}{|c|c|c|}
\hline
$x_{1:2}$ & Discrete renewal $A_{h}(k,k)$ & Monte Carlo \\
\hline
(1,1)   & 0.2497 & 0.2502 \\
(2,2)   & 0.1276 & 0.1263 \\
(4,4)   & 0.0870 & 0.0841 \\
(8,8)   & 0.0539 & 0.0496 \\
(16,16) & 0.0288 & 0.0268 \\
\hline
\end{tabular}
\caption{Comparison of continuous-time availability estimates along the diagonal.}
\label{tab:avail-mc}
\end{table}

\paragraph{Computational remark (link to FFT methods).}
Evaluating $A_h$ or $M_h$ on a full $K\times K$ grid requires repeated 2D convolutions.
A direct implementation scales on the order of $\mathcal{O}(K^4)$ operations, whereas FFT-based convolution
(Remark~\ref{rem:fft-cost} and Appendix~\ref{app:fft}) reduces each convolution to about
$\mathcal{O}(K^2\log K)$ operations, which is crucial for large grids induced by small step sizes $h_1,h_2$.

\section{Discussion}\label{sec:discussion}

This paper develops a discrete-time renewal theory on the multi-index lattice $\mathbb{N}^d$ and shows that several classical renewal identities admit multi-time analogues when the algebra is formulated in terms of multi-index convolution. A key point is that multi-time renewal equations are linear convolution equations in a commutative ring of sequences, and therefore admit explicit solutions whenever the corresponding convolutional inverse exists. This formulation gives a common treatment of renewal functions and of reliability/availability functionals, and provides a coefficient-level counterpart to the transform-based and integral-equation approaches used in the continuous-time bivariate literature \cite{hunter1974renewal,hunter1974renewal1,hunter1977renewal}.

From a computational point of view, the multi-index formulation is useful because it leads to effective finite-grid computations. Direct evaluation of $d$-dimensional convolutions becomes expensive as the grid size increases, whereas FFT-based convolution reduces the cost substantially, up to logarithmic factors. Similarly, computing convolutional inverses by coefficientwise recursion becomes costly on large grids, while Newton-type reciprocal iteration doubles the truncation order at each step. Thus FFT-accelerated products and Newton-type inversion give an effective numerical method for multi-time renewal computations on grids where direct summation is not feasible. In the bivariate continuous-time setting, several approximation methods have been proposed for the renewal function \cite{hadji2015two,arunachalam2015approximation}; the present results show how grid discretizations lead to finite convolution equations and to inverse problems on the lattice, with no restriction to two coordinates.

On the asymptotic side, the proportional-growth $k/|k|_1\to\lambda$ yields directional limits that extend classical renewal strong laws to multi-time indices. The CLT for additive functionals gives fluctuations of cycle-accumulated performance measures under multi-time truncation. The counting process $N(k)$ itself is the minimum of strongly dependent marginal counts, and a Gaussian limit for $N(k)$ is obtained under an additional uniqueness condition, namely when a single rate-determining coordinate attains the minimum. The case where several coordinates attain the same limiting rate remains a natural problem. In such a case, the limiting fluctuations are expected to involve minima of correlated Gaussian variables, in agreement with the role of joint bivariate asymptotics in Hunter's two-dimensional theory \cite{hunter1974renewal1} and with multivariate renewal fluctuation results \cite{niculescu1984asymptotic}. A remaining question is to combine joint functional CLTs for vectors of marginal renewal counts with the minimum map in order to obtain the corresponding non-normal limits.

We also considered inference from fixed-horizon observations. Observing the completed increments together with the multi-time vector $U_m$ induces a multivariate censoring event: the next increment exceeds the remaining budget in at least one coordinate. This leads to an explicit exact nonparametric MLE, which differs from the empirical estimator only by a term of order $\mathcal{O}(1/N(m))$ and is asymptotically normal under the same scaling regime. This complements existing inferential work in multi-time renewal theory that focuses on estimating renewal functions themselves \cite{harel2019asymptotic}. Further statistical questions include parametric families, goodness-of-fit procedures, and estimation from several independent trajectories. Another direction is to combine the fixed-horizon likelihood with weighted renewal functionals, as in multi-attribute replacement models \cite{mallor2007multivariate}, where the input term in the renewal equation is also estimated from data.

Finally, the discrete-to-continuous approximation considered here can be refined by studying discretization error as a function of the mesh size and of the regularity of the underlying continuous distribution. More generally, the same idea extends to higher-dimensional continuous-time renewal equations, where the renewal operator is a multivariate Riemann--Stieltjes convolution and explicit formulas are rarely available beyond special cases \cite{hunter1974renewal1}. Error bounds for high-dimensional grid schemes, as well as adaptive discretizations near the coordinate axes, would strengthen the connection between continuous-time models and their lattice approximations.

Beyond the i.i.d. setting, the same convolutional calculus extends to matrix-valued kernels arising in Markov-modulated or state-dependent multi-time renewal models. In that case, scalar renewal kernels are replaced by matrix-valued kernels, and the renewal equations become linear equations in a matrix-convolution algebra. The Fourier transform can still be applied entrywise, with matrix multiplication in the transformed domain. This could lead to a very interesting extension covering multi-time Markov renewal and semi-Markov models, in both discrete time and discretized continuous time.

Overall, the results give a discrete multi-time renewal theory with coefficient formulas, limit theorems, likelihood estimators, and numerical procedures. The approach is especially relevant for stochastic systems in which several exposure variables determine the renewal mechanism, in particular in reliability, maintenance, and warranty analysis.
\bibliographystyle{plainnat}
\bibliography{ref_arxiv}

\begin{appendices}
   \section{Bivariate distributions}\label{app:bivariate-distributions}

This appendix collects the bivariate distributions used in the numerical illustrations.
We write $x=(x_1,x_2)$, $\mathbb{R}_+:=[0,\infty)$, and $\mathbb{N}_0:=\{0,1,2,\ldots\}$.

\subsection{Continuous models}

\noindent The first continuous bivariate model used in the computations is the following.
\begin{definition}\label{def:downton-exp}
Let $\lambda_1,\lambda_2>0$ and $0\le \rho<1$.
A random vector $(X_1,X_2)$ taking values in $\mathbb{R}_+^2$ is said to follow the
Moran--Downton bivariate exponential distribution if it has joint density
\begin{equation}\label{eq:downton-pdf}
f_{X_1,X_2}(x_1,x_2)
=
\frac{\lambda_1\lambda_2}{1-\rho}
\exp\!\left(-\frac{\lambda_1 x_1+\lambda_2 x_2}{1-\rho}\right)\,
I_0\!\left(\frac{2\sqrt{\rho\,\lambda_1\lambda_2\,x_1x_2}}{1-\rho}\right),
\qquad x_1,x_2\in\mathbb{R}_+,
\end{equation}
where $I_0$ is the modified Bessel function of the first kind of order $0$,
\[
I_0(z)=\sum_{m=0}^\infty \frac{1}{(m!)^2}\left(\frac{z}{2}\right)^{2m}.
\]
Under the rate parametrization used in \eqref{eq:downton-pdf}, the marginals are exponential
with rates $\lambda_1$ and $\lambda_2$, and $\rho=0$ yields independence.
(Alternative scale parametrizations appear in the literature; see, e.g.,
\cite{downton1970bivariate,hanandeh2013inference}.)
\end{definition}

\noindent The second continuous bivariate model is the Farlie--Gumbel--Morgenstern exponential distribution.
\begin{definition}\label{def:fgm-exp}
Let $\lambda_1,\lambda_2>0$ and $\rho\in[-1,1]$.
Let
\[
F_X(x_1)=1-e^{-\lambda_1 x_1},\qquad
F_Y(x_2)=1-e^{-\lambda_2 x_2},\qquad x_1,x_2\in\mathbb{R}_+.
\]
The FGM bivariate exponential distribution is defined by the joint cdf
\begin{equation}\label{eq:fgm-cdf}
F(x_1,x_2)
=
F_X(x_1)F_Y(x_2)\Big[1+\rho\,(1-F_X(x_1))(1-F_Y(x_2))\Big],
\qquad (x_1,x_2)\in\mathbb{R}_+^2,
\end{equation}
whose joint density is
\begin{equation*}
f(x_1,x_2)
=
\lambda_1\lambda_2\,e^{-(\lambda_1 x_1+\lambda_2 x_2)}
\Big[1+\rho\,(1-2e^{-\lambda_1 x_1})(1-2e^{-\lambda_2 x_2})\Big],
\qquad x_1,x_2\in\mathbb{R}_+.
\end{equation*}
Equivalently, \eqref{eq:fgm-cdf} is obtained by combining the exponential marginals
with the FGM copula \cite{morgenstern1956einfache,nelsen2006introduction}.
The case $\rho=0$ corresponds to independence.
\end{definition}

\subsection{Discrete model}\label{subsec:DISC}

\noindent The following finite-difference identity recovers a discrete pmf from a survival function.
\begin{remark}\label{rem:survival-to-pmf}
Let $X=(X_1,X_2)$ be an $\mathbb{N}_0^2$-valued random vector and define the joint survival function
\[
S(k_1,k_2):=\mathbb{P}(X_1\ge k_1,\;X_2\ge k_2),\qquad (k_1,k_2)\in\mathbb{N}_0^2.
\]
Then the joint pmf $p(k_1,k_2)=\mathbb{P}(X_1=k_1,X_2=k_2)$ is obtained by the mixed finite difference
\begin{equation}\label{eq:pmf-from-survival}
p(k_1,k_2)
=
S(k_1,k_2)-S(k_1+1,k_2)-S(k_1,k_2+1)+S(k_1+1,k_2+1),
\qquad k_1,k_2\in\mathbb{N}_0.
\end{equation}
\end{remark}

\noindent The discrete bivariate Weibull model used in the numerical section is defined as follows.
\begin{definition}\label{def:bdw}
Let $q_1,q_2\in(0,1)$, $\beta_1,\beta_2>0$, and $\theta\in(0,1]$.
A random vector $X=(X_1,X_2)$ on $\mathbb{N}_0^2$ is said to follow the bivariate discrete Weibull
model (as used in this paper) if its joint survival function is
\begin{equation*}
S(k_1,k_2)
=
\mathbb{P}(X_1\ge k_1,\;X_2\ge k_2)
=
q_1^{k_1^{\beta_1}}\,q_2^{k_2^{\beta_2}}\,\theta^{\,k_1^{\beta_1}k_2^{\beta_2}},
\qquad (k_1,k_2)\in\mathbb{N}_0^2,
\end{equation*}
and the pmf is computed from $S$ via \eqref{eq:pmf-from-survival}.
The marginals satisfy $S(k_1,0)=q_1^{k_1^{\beta_1}}$ and $S(0,k_2)=q_2^{k_2^{\beta_2}}$,
so each component is a (univariate) discrete Weibull distribution in the sense of
Nakagawa--Osaki \cite{nakagawa1975discrete}.
The case $\theta=1$ yields independence (product survival).
\end{definition}

\noindent The parameters of this discrete model must satisfy at least the following admissibility condition.
\begin{remark}\label{rem:bdw-admissibility}
Under the convention in Remark~\ref{rem:survival-to-pmf}, nonnegativity of the mass at $(0,0)$ implies
\begin{equation*}
p(0,0)=1-q_1-q_2+q_1q_2\theta \ge 0
\qquad\Longleftrightarrow\qquad
\theta \ge \frac{q_1+q_2-1}{q_1q_2}.
\end{equation*}
More generally, parameters should be chosen so that $p(k_1,k_2)\ge 0$ for all $(k_1,k_2)\in\mathbb{N}_0^2$.
\end{remark}

   \section{Fast Fourier transform and multidimensional convolution}\label{app:fft}

This appendix summarizes the Fourier-transform identities used to accelerate multidimensional convolutions.
We distinguish between (i) the \emph{discrete-time Fourier transform} (DTFT) of an infinite sequence and (ii) the finite \emph{discrete
Fourier transform} (DFT) computed by the FFT.
Standard references include \cite{cooley1965algorithm,van1992computational}.

\subsection{Notation}
Let $\mathrm{i}:=\sqrt{-1}$.
For $\omega=(\omega_1,\ldots,\omega_d)\in\mathbb{R}^d$ and $k=(k_1,\ldots,k_d)\in\mathbb{Z}^d$, write
$k\cdot \omega := \sum_{u=1}^d k_u\,\omega_u$.
For a vector $N=(N_1,\ldots,N_d)\in\mathbb{N}^d$, define the index set
\[
I_N := \{0,\ldots,N_1-1\}\times\cdots\times\{0,\ldots,N_d-1\}.
\]
For $k\in\mathbb{Z}^d$, define componentwise reduction modulo $N$ by
\[
k \bmod N := (\,k_1\bmod N_1,\ldots,k_d\bmod N_d\,)\in I_N.
\]

\subsection{DTFT (infinite sequences) and convolution}
Let $a:\mathbb{Z}^d\to\mathbb{C}$ be absolutely summable, $a\in\ell^1(\mathbb{Z}^d)$ (in particular, any finitely supported sequence).
Its \emph{discrete-time Fourier transform} is the $2\pi$-periodic function
\[
\widehat a(\omega) := \sum_{k\in\mathbb{Z}^d} a(k)\,e^{-\mathrm{i}\,k\cdot\omega},
\qquad \omega\in[-\pi,\pi]^d.
\]
The inversion formula holds (in the $\ell^1$ setting) as
\[
a(k) = \frac{1}{(2\pi)^d}\int_{[-\pi,\pi]^d} \widehat a(\omega)\,e^{\mathrm{i}\,k\cdot\omega}\,d\omega,
\qquad k\in\mathbb{Z}^d.
\]
If $c=a*b$ denotes the (linear) convolution on $\mathbb{Z}^d$,
\[
(a*b)(k):=\sum_{\ell\in\mathbb{Z}^d} a(\ell)\,b(k-\ell),
\]
then the convolution theorem states
\[
\widehat{a*b}(\omega)=\widehat a(\omega)\,\widehat b(\omega).
\]

\subsection{Finite DFT and circular convolution}
For numerical computation we work with finite arrays supported on $I_N$.
Let $a:I_N\to\mathbb{C}$. Its $d$-dimensional \emph{DFT} is defined by
\begin{equation*}
\widehat a(m)
:= \sum_{k\in I_N} a(k)\,
\exp\!\left\{-2\pi \mathrm{i}\sum_{u=1}^d \frac{k_u m_u}{N_u}\right\},
\qquad m\in I_N.
\end{equation*}
The inverse DFT is
\begin{equation*}
a(k)
= \frac{1}{\prod_{u=1}^d N_u}\sum_{m\in I_N} \widehat a(m)\,
\exp\!\left\{2\pi \mathrm{i}\sum_{u=1}^d \frac{k_u m_u}{N_u}\right\},
\qquad k\in I_N.
\end{equation*}
\noindent The finite DFT corresponds to the following circular convolution.
\begin{definition}\label{def:circ-conv}
For $a,b:I_N\to\mathbb{C}$, their circular convolution $a*_N b:I_N\to\mathbb{C}$ is
\begin{equation*}
(a*_N b)(k)
:= \sum_{\ell\in I_N} a(\ell)\, b\big((k-\ell)\bmod N\big),
\qquad k\in I_N.
\end{equation*}
\end{definition}

With these definitions, the DFT convolution theorem holds:
\[
\widehat{a*_N b}(m)=\widehat a(m)\,\widehat b(m),
\qquad m\in I_N.
\]
Therefore, $a*_N b$ can be computed by two forward FFTs, pointwise multiplication in the frequency domain, and one inverse FFT.

\subsection{Linear convolution via zero-padding}
Let $a$ and $b$ be (finite) sequences supported on boxes of sizes $N=(N_1,\ldots,N_d)$ and $M=(M_1,\ldots,M_d)$, respectively.
To compute their \emph{linear} convolution on the full support $\{0,\ldots,N_1+M_1-2\}\times\cdots\times\{0,\ldots,N_d+M_d-2\}$,
choose padding lengths $L=(L_1,\ldots,L_d)$ satisfying
\[
L_u \ge N_u+M_u-1,\qquad u=1,\ldots,d,
\]
embed $a$ and $b$ into arrays on $I_L$ by zero-padding, compute the circular convolution on $I_L$, and then restrict the output to the
valid (non-aliased) index range. This is the standard FFT-based method for linear convolution.

\subsection{Complexity}
Let $L=(L_1,\ldots,L_d)$ be the FFT grid size and $N_L:=\prod_{u=1}^d L_u$.
A separable $d$-dimensional FFT (successive 1D FFTs along each dimension) has computational complexity
\[
\mathcal{O}\!\left(N_L \sum_{u=1}^d \log L_u\right),
\]
and thus FFT-based convolution on $I_L$ has the same order of complexity \cite{cooley1965algorithm,van1992computational}.

\noindent The same Fourier identities apply componentwise to matrix-valued sequences.
\begin{remark}
If $a(k)$ and $b(k)$ are matrices (of fixed size) and convolution is defined with matrix multiplication,
\[
(a*b)(k)=\sum_{\ell} a(\ell)\,b(k-\ell),
\]
then the DFT is applied entrywise and the convolution theorem remains valid with \emph{matrix products} in the frequency domain:
$\widehat c(m)=\widehat a(m)\,\widehat b(m)$, $m\in I_N$.
\end{remark}

\section{Newton's inversion method: proof of the update}\label{app:newton-proof}

We work in the commutative ring of multivariate formal power series
$\mathbb{R}[[x_1,\ldots,x_d]]$ and denote by
\[
\mathfrak{m}:=\langle x_1,\ldots,x_d\rangle
\]
its maximal ideal. Recall that $P_a$ is invertible in $\mathbb{R}[[x_1,\ldots,x_d]]$ if and only if its constant term is nonzero,
i.e.\ $a(0_d)\neq 0$.

Assume that at iteration level $N$ we have a truncated approximation $P_{b,N}$ satisfying
\begin{equation}\label{eq:newton-assumption}
P_a\,P_{b,N}=1 \quad \bmod \ \mathfrak{m}^{2^N}.
\end{equation}
Equivalently, the error term
\begin{equation}\label{eq:newton-error}
e_N:=1-P_a\,P_{b,N}
\end{equation}
belongs to $\mathfrak{m}^{2^N}$, i.e.\ all monomials appearing in $e_N$ have total degree at least $2^N$.

Define the Newton update
\begin{equation}\label{eq:newton-update-app}
P_{b,N+1}
:=P_{b,N}+P_{b,N}\,e_N
= P_{b,N}\,\bigl(2-P_aP_{b,N}\bigr)
\quad \bmod \ \mathfrak{m}^{2^{N+1}},
\end{equation}
which is the multivariate analogue of the classical reciprocal-series Newton iteration
(see, e.g., \cite{brent1978fast,von2013modern}).

\medskip
\noindent\textbf{Claim.}
The updated approximation satisfies
\[
P_a\,P_{b,N+1}=1 \quad \bmod \ \mathfrak{m}^{2^{N+1}}.
\]

\begin{proof}
Starting from \eqref{eq:newton-update-app} and using \eqref{eq:newton-error},
\begin{align*}
P_a\,P_{b,N+1}
&= P_a\bigl(P_{b,N}+P_{b,N}e_N\bigr)
 = P_aP_{b,N} + (P_aP_{b,N})e_N \\
&= (1-e_N) + (1-e_N)e_N
 = 1 - e_N^2.
\end{align*}
By assumption \eqref{eq:newton-assumption}, we have $e_N\in \mathfrak{m}^{2^N}$, and since
$\mathfrak{m}^{r}\,\mathfrak{m}^{s}\subseteq \mathfrak{m}^{r+s}$ in any commutative ring, it follows that
\[
e_N^2 \in \mathfrak{m}^{2^N}\cdot \mathfrak{m}^{2^N}\subseteq \mathfrak{m}^{2^{N+1}}.
\]
Therefore $P_aP_{b,N+1}=1 \bmod \mathfrak{m}^{2^{N+1}}$, proving the claim.
\end{proof}

\noindent The congruence in Newton inversion has the following coefficientwise meaning.
\begin{remark}\label{rem:newton-meaning}
The congruence $P_aP_{b,N}=1 \bmod \mathfrak{m}^{2^N}$ means that $P_{b,N}$ coincides with the exact inverse
$P_a^{-1}$ in all coefficients of monomials $x^k$ with total degree $|k|_1<2^N$.
Thus each Newton step doubles the truncation order (quadratic error reduction):
the new error is $e_{N+1}=e_N^2 \in \mathfrak{m}^{2^{N+1}}$.
\end{remark}

\end{appendices}

\end{document}